\documentclass[10pt]{article}
\usepackage{mathrsfs}
\usepackage{amsfonts}
\usepackage{amsthm,amsmath,amssymb,anysize}
\newtheorem{lemma}{Lemma}[section]
\newtheorem{theorem}[lemma]{Theorem}
\newtheorem{remark}[lemma]{Remark}
\newtheorem{proposition}[lemma]{Proposition}

\newtheorem{corollary}[lemma]{Corollary}
\newtheorem{convention}[lemma]{Convention}
\usepackage[numbers,sort&compress]{natbib}
\marginsize{4.2cm}{4.2cm}{3.6cm}{3cm}%×óÓÒÉÏÏÂ
\numberwithin{equation}{section}

\title{\textsf{Derivations  of the finite-dimensional special odd Hamiltonian  superalgebras}}
\author{\textsc{Wei Bai$^{1, 2}$}\;  \textsc{ Wende Liu$^{1, 2}$}
\footnote{Correspondence: wendeliu@ustc.edu.cn (W. Liu)} \footnote{Supported by the NSF of China (10871057)
and the NSF of HLJ Province, China (A200802)}\;  \textsc{Lan Ni $^{3}$} \\
  \\
\textit{$^{1}$Department of Mathematics},
\textit{Harbin Institute of Technology}\\
\textit{Harbin 150006, China}\\
\ \ \textit{$^{2}$School of Mathematical Sciences},
\textit{Harbin Normal University} \\
\textit{Harbin 150025, China}\\
\ \ \textit{$^{3}$Department of Mathematics and Mechanics},\\
\textit{Heilongjiang Institute of Science and Technology} \\
\textit{Harbin 150027, China}  }
\date{ }
\begin{document}
\maketitle
\begin{quotation}
\noindent\textbf{Abstract} The aim is to determine the derivations
of the three series of finite-dimensional $\mathbb{Z}$-graded Lie superalgebras of Cartan-type
over a field of characteristic $p>3$, called the
special odd Hamiltonian  superalgebras. To that end we first determine
the derivations of negative $\mathbb{Z}$-degree
for the restricted and simple special odd Hamiltonian superalgebras by
means of weight space decompositions. Then the results
are used to determine the  derivations  of negative $\mathbb{Z}$-degree
 for the non-restricted and non-simple special odd Hamiltonian  superalgebras.
 Finally the derivation algebras and the outer derivation
 algebras of those Lie superalgebras are completely determined.\\

\\

  \end{quotation}

  \setcounter{section}{-1}
\section{Introduction}

 We work over a field $\mathbb{F}$ of positive characteristic. Using the divided powers  algebras
 instead of  the polynomial algebras one can construct eight families of finite dimensional $\mathbb{Z}$-graded Lie superalgebras of Cartan-type over  $\mathbb{F}$, which are analogous to the  vectorial Lie superalgebras  over  $\mathbb{C}$
 (see \cite{bl,k1,k2,z}, for example). All these Lie superalgebras are  subalgebras of the full
 (super)derivation algebras  of the tensor products  of the finite dimensional divided algebras and the exterior algebras, which are viewed as associative superalgebras in the obvious fashion. The derivation algebras were sufficiently studied for the  modular Lie superalgebras of Cartan-type mentioned above (see \cite{fzj,ly1,lz3,lzw,wz,zz}), except the so-called  special odd Hamiltonian  superalgebras (see \cite{lh}).

 The present paper aims to determine the derivation algebras of the  special odd Hamiltonian  superalgebras, especially, the outer derivation algebras. Our work is heavily depend on the results obtained in \cite{lh} and contains certain results obtained in 2005 in the thesis for master-degree by the third-named author \cite{n}. We should mention that we use the method  for Lie algebras \cite{st} and benefit much from reading \cite{sf,st}.

\section{Preliminaries}

Hereafter $\mathbb{F}$ is a field of characteristic $p>3$;
$\mathbb{Z}_2:= \{\overline{0},\overline{1}\}$ is the field
of two elements.
For a vector superspace $V=V_{\bar{0}}\oplus V_{\bar{1}}$, we denote by $\mathrm{p}(a)=\theta$
the parity of a homogeneous element $a\in V_{\bar{\theta}}$, $\bar{\theta}\in\mathbb{Z}_2$. We
assume throughout that the notation $\mathrm{p}(x)$ implies that $x$ is a $\mathbb{Z}_2$-homogeneous element.
$\mathbb{N}$ and $\mathbb{N}_0$ are the sets of
positive integers and nonnegative integers, respectively. Let $m\geq 3$ denote
a fixed positive integer and $\mathbb{N}^m$ the additive monoid of
 $m$-tuples of nonnegative integers.  Fix two $m$-tuples
$\underline t:=(t_1,\ldots,t_m)\in \mathbb{N}^m$ and
$\pi:=(\pi_1,\ldots,\pi_m)\in \mathbb{N}^m,$ where $\pi_i:=p^{t_i}-1.$ Let
$\mathcal{O}(m;\underline{t})$ be the divided power algebra   with
$\mathbb{F}$-basis $\{x^{(\alpha)}\mid \alpha\in
\mathbb{A}(m;\underline{t})\}$, where
$\mathbb{A}(m;\underline{t}):=\{\alpha\in \mathbb{N}_0^m\mid
\alpha_i\leq \pi_i\}$. Write $|\alpha|:=\Sigma_{i=1}^m\alpha_i$.
For $\varepsilon_i:=(\delta_{i1},\delta_{i2},\ldots,\delta_{im})\in \mathbb{A}(m;\underline{t})$,
 we usually write $x_i$ for $x^{(\varepsilon_i)}$, where $i=1,\ldots,m.$ Let
$\Lambda(m)$ be the  exterior superalgebra over  $\mathbb{F}$ in $m$ variables
$x_{m+1},x_{m+2},\ldots,x_{2m}$.
The tensor product
$\mathcal{O}(m,m;\underline{t})
:=\mathcal{O}(m;\underline{t})\otimes_{\mathbb{F}} \Lambda(m)$
 is an associative super-commutative
superalgebra with a $\mathbb{Z}_2$-grading structure induced by the
trivial $\mathbb{Z}_2$-grading of $\mathcal{O}(m;\underline{t})$ and
the standard $\mathbb{Z}_2$-grading of $\Lambda (m).$   For $g\in
\mathcal{O}(m,\underline{t})$, $f\in \Lambda(m),$   write
$gf $ for $ g\otimes f$. Note that $ x^{(\alpha) }x^{(\beta)}
=\binom{\alpha +\beta }{ \alpha}
 x^{( \alpha +\beta)}$ for $\alpha,\beta \in
\mathbb{N}^m, $ where $\binom{ \alpha +\beta}
{\alpha}:=\prod_{i=1}^m\binom{ \alpha _i+\beta _i}{ \alpha _i}.$
Let
$$\mathbb{B}(m):=\left\{ \langle i_1,i_2,\ldots,i_k\rangle \mid m+1\leq i_1<i_2<\cdots <i_k\leq 2m;\;
k\in \overline{0, m}\right\}.$$
 For
 $u:=\langle i_1,i_2,\ldots,i_k\rangle \in \mathbb{B}(m),$ write $ |u| :=k$
 and  $x^u:=x_{i_1}x_{i_2}\cdots x_{i_k}.$
  Notice that we also  denote the index set $\{i_1,i_2,\ldots,i_k\}$ by $u$ itself.
  For $u,\upsilon \in \mathbb{B}(m)$ with $u\cap
  \upsilon=\emptyset,$ define $u+\upsilon$ to be the unique  element $w\in \mathbb{B}(m)$ such that $ w=u\cup
  \upsilon.$ Similarly, if $\upsilon\subset u,$ define $u-\upsilon$ to be the unique element $w\in
  \mathbb{B}(m)$ such that $w=u\setminus\upsilon.$
Write $\omega=\langle m+1,\ldots, 2m\rangle$.
  Note that
$\mathcal{O}(m,m;\underline{t})$ has a standard
$\mathbb{F}$-basis $\{ x^{( \alpha ) }x^u\mid(\alpha,u) \in
\mathbb{A}\times \mathbb{B}\}. $ Put $\mathbf{Y_{0}}:=\overline{1,m},$
$\mathbf{Y_{1}}:=\overline{m+1,2m}$ and $\mathbf{Y}:=\overline{1,2m}.$ Let
$\partial_r$ be the superderivation of
$\mathcal{O}(m,m;\underline{t}) $ such that
 \[ \partial_{r}(x^{(\alpha)}x^u):=
 \left\{\begin{array}{ll} x^{(\alpha-\varepsilon_r)}x^u, &
r\in\mathbf{Y_0}
\\
x^{(\alpha)}\partial x^u/\partial x_r, & i\in \mathbf{Y_1}.
\end{array}\right.\]
The generalized Witt superalgebra  $
W\left(m,m;\underline{t}\right)$ is spanned by all $f_r \partial_r,$
where $ f_r\in \mathcal{O}(m,m;\underline{t}),$ $r\in \mathbf{Y}.$   Note that
$W(m,m;\underline{t} ) $ is a free $ \mathcal{O} \left(
m,m;\underline{t}\right)$-module with basis $ \{ \partial_r\mid r\in
\mathbf{Y} \}.$ In particular, $W ( m,m;\underline{t} ) $ has a so-called
standard $\mathbb F$-basis
 $\{x^{(\alpha)}x^{u}\partial_r\mid (\alpha,u,r)\in
\mathbb{A}\times\mathbb{B}\times \mathbf{Y}\}.$
%One may verify that
%\[[fD, gE]=fD(g)E-(-1)^{{\mathrm
%p}(fD){\mathrm p}(gE)}gE(f)D+(-1)^{{\mathrm p}(D){\mathrm p}(g)}fg[D, E]\] for
%$ f,g\in \mathcal{O} (m,m;\underline{t}) $ and $ D,E\in {\mathrm Der}\,\mathcal{O} (m,m;\underline{t}).$
Note that $ \mathcal{O} \left( m,m;\underline{t}\right) $ possesses
a so-called standard ${\mathbb Z}$-grading structure
$\mathcal{O} \left(m,m;\underline{t}\right)=\bigoplus _{r=0}^{\xi}
\mathcal{O}(m,m;\underline{t})_{r}$ by letting
        $$\mathcal{O}(m,m;\underline{t})_{r}:=
        \mathrm{span} _{\mathbb F}\{ x^{(\alpha)} x^{u}\mid |\alpha|+|u|=r \},
        \quad \xi:=|\pi|+m=\sum_{i\in \mathbf{Y}_0}p^{t_{i}}.$$
 This
induces naturally   a $\mathbb{Z}$-grading structure, also called standard,
  $$W(m,m;\underline{t})
=\bigoplus_{i=-1}^{\xi-1} W(m,m;\underline{t})_{i},$$
  where
$$W(m,m;\underline{t})_{i} :=\mathrm{span}_{\mathbb{F}}\{f\partial_{r}\mid
r\in \mathbf{Y},\ f\in \mathcal{O}(m,m;\underline{t})_{i+1}\}.$$
Put
\[i':= \left\{\begin{array}{ll}
i+m, &\mbox{if}\; i\in\mathbf{Y_0}
\\i-m, &\mbox{if}\; i\in \mathbf{Y_1},
\end{array}\right.\quad \mu(i):=\left\{\begin{array}{ll}\bar{0},  &\mbox{if}\; i\in \mathbf{Y_0}\\
\overline{1},&\mbox{if}\; i\in\mathbf{Y_1}. \end{array}\right. \]
Clearly, $\mathrm{p}(\partial_i)=\mu(i)$.
 Define the linear operator
  $\mathrm{T_H}:\mathcal{O}(m,m;\underline{t})\longrightarrow
 W(m,m;\underline{t})$ such that
 \[
 \mathrm{T_H}(a):=\sum_{i\in \mathbf{Y}}(-1)^{\mathrm{p}(\partial_i)\mathrm{p}(a)}\partial_i(a)\partial_{i'}
  \quad\mbox{for}\  a\in
 \mathcal{O}(m,m;\underline{t}).
 \]
 Note that $\mathrm{T_H}$ is  odd with respect to the $\mathbb{Z}_2$-grading  and has degree $-2$
 with respect to the $\mathbb{Z}$-grading. The following formula is well known:
\begin{equation*}\label{hee1.1}
   [\mathrm{T_H}(a),\mathrm{T_H}(b)]=\mathrm{T_H}(\mathrm{T_H}(a)(b)) \quad \mbox{for}\,
   a,b\in\mathcal{O}(m,m;\underline{t})
\end{equation*}
and
\[HO(m,m;\underline{t}):=\{\mathrm{T_H}(a)\,|\,a\in\mathcal{O}(m,m;\underline{t})\}
\]
is a finite-dimensional simple Lie superalgebra, called the
 odd Hamiltonian superalgebra \cite{k2,lz2}.
Put
$$\overline{HO}(m,m;\underline{t}):=\overline{HO}(m,m;\underline{t})_{\bar0}\oplus
\overline{HO}(m,m;\underline{t})_{\bar1}$$
 where for $\alpha \in \mathbb{Z }_2,$
\begin{eqnarray*}
 \overline{HO}(m,m;\underline{t})_{\alpha}:&=&\bigg \{\sum_{i\in \mathbf{Y}} a_i\partial_i\in
  W(m,m;\underline{t})_{\alpha}\bigg|\\
 &&\partial_i(a_{j'})=(-1)^{\mu(i)\mu(j)+(\mu(i)+\mu(j))(\alpha+\bar{1})}\partial_j(a_{i'}), i, j \in \mathbf{Y}\bigg\}.
\end{eqnarray*}
We state certain basic results in \cite[Proposition 1]{lzw}, which will be used in the following sections:
\begin{itemize}
\item[$\mathrm{(i)}$]
 Both $HO(m,m;\underline{t})$ and $\overline{HO}(m,m;\underline{t})$ are ${\mathbb Z}$-graded
subalgebras of $ W(m,m;\underline{t})$,
\begin{eqnarray*}
 HO(m,m;\underline{t})=\bigoplus_{i=-1}^{\xi-2}HO(m,m;\underline{t})_i; \quad
\overline{HO}(m,m;\underline{t})=\bigoplus_{i=-1}^{\xi-1}\overline{HO}(m,m;\underline{t})_i.
\end{eqnarray*}

\item[$\mathrm{(ii)}$]  $HO(m,m;\underline{t})$ is a $\mathbb{Z}$-graded ideal of $\overline{HO}(m,m;\underline{t})$.\\

\item[$\mathrm{(iii)}$]  $\ker(\mathrm{T_H})=\mathbb{F}\cdot 1$.\\
\end{itemize}
Let $\mathrm{div}: W(m,m;\underline{t})\longrightarrow \mathcal{O}(m,m;\underline{t})$
be the  {divergence}, which is a linear mapping such that
$$\mathrm{div}(f_r\partial_r)=(-1)^{\mathrm{p}(\partial_r)\mathrm{p}(f_r)}\partial_r(f_r)\quad \mbox{for all}\, r\in\mathbf{Y}.$$
 Note that $\mathrm{div}$ is an even $\mathbb{Z}$-homogeneous superderivation of $W(m,m;\underline{t})$
 into the module $\mathcal{O}(m,m;\underline{t})$, that is
\begin{eqnarray}\label{bl1305}
\mathrm{div}[E,D]=E(\mathrm{div} D)-(-1)^{\mathrm{p}(E)\mathrm{p}(D)}D(\mathrm{div } E)\quad\mbox{for all}\ D, E\in W(m,m;\underline{t}).
\end{eqnarray}
Putting
\begin{eqnarray*}
&&S'(m,m;\underline{t}):=\{D \in W(m,m;\underline{t})\mid\mathrm{div}(D)=0\},\\
&&\overline{S}(m,m;\underline{t}):=\{D \in W(m,m;\underline{t})\mid\mathrm{div}(D)\in \mathbb{F}\},
\end{eqnarray*}
 we have:
\begin{itemize}
\item[$\mathrm{(i)}$]
   Both $S'(m,m;\underline{t})$ and $\overline{S}(m,m;\underline{t})$ are ${\mathbb Z}$-graded
subalgebras of $ W(m,m;\underline{t})$:
\begin{eqnarray*}
S'(m,m;\underline{t})=\bigoplus_{i=-1}^{\xi-1}S'(m,m;\underline{t})_i; \quad
\overline{S}(m,m;\underline{t})=\bigoplus_{i=-1}^{\xi-1}\overline{S}(m,m;\underline{t})_i.
\end{eqnarray*}
\item[$\mathrm{(ii)}$]  $S'(m,m;\underline{t})$ is
a $\mathbb{Z}$-graded ideal of  $\overline{S}(m,m;\underline{t})$.
\end{itemize}
Here we write down the following symbols which will be frequently used in the future:
\begin{eqnarray*}
&&\Delta:=\sum_{i=1}^m\Delta_i, \quad
\Delta_{i}:=\partial_{i}\partial_{i'}\quad\mbox{for}\, i\in \mathbf{Y_0};\\
&&\nabla_{i}(x^{(\alpha)}x^{u}):=x^{(\alpha+\varepsilon_{i})}x_{i'}x^{u}\quad \mbox{for}\,(\alpha,u) \in
\mathbb{A}\times \mathbb{B}, i\in \mathbf{Y_0};\\
&&\Gamma_{i}^{j}:=\nabla_{j}\Delta_{i}\quad \mbox{for}\; i,j\in \mathrm{Y_0};\\
&&\mathbf{I}(\alpha,u):=\{i\in  \mathbf{Y_0}\mid
\Delta_{i}(x^{(\alpha)}x^{u})\neq0\};\\
&&\widetilde{\mathbf{I}}(\alpha,u):=\{i\in
\mathbf{Y_0}\mid \nabla_{i}(x^{(\alpha)}x^{u})\neq0\};\\
&&\mathcal{D}^{\ast}:=\{x^{(\alpha)}x^{u}\mid
\mathbf{I}(\alpha,u)\neq\emptyset,\widetilde{\mathbf{I}}(\alpha,u)\neq\emptyset\};\\
&&\mathcal{D}_1:=\{x^{(\alpha)}x^{u}\mid
\mathbf{I}(\alpha,u)=\widetilde{\mathbf{I}}(\alpha,u)=\emptyset\};\\
&&\mathcal{D}_2:=\{x^{(\alpha)}x^{u}\mid
\mathbf{I}(\alpha,u)=\emptyset, \, \widetilde{\mathbf{I}}(\alpha,u)\neq\emptyset\}.
\end{eqnarray*}
In this paper we mainly study  the three series of Lie superalgebras:
\begin{eqnarray*}
&&SHO(m,m;\underline{t}):=S'(m,m;\underline{t})\cap HO(m,m;\underline{t}),\\
&& SHO(m,m;\underline{t})^{(1)}:=[SHO(m,m;\underline{t}), SHO(m,m;\underline{t})],\\
&& SHO(m,m;\underline{t})^{(2)}:=[SHO(m,m;\underline{t})^{(1)}, SHO(m,m;\underline{t})^{(1)}],
\end{eqnarray*}
called the  {special odd Hamiltonian superalgebras}.
By \cite [Theorem 4.1]{lh} they are centerless and
$SHO(m,m;\underline{t})^{(2)}$ is simple. Further informations for these Lie superalgebras can be found in \cite{k2,lh}.

\begin{convention}\label{c1} For short, we usually omit the parameter $(m,m;\underline{t})$ and write
$\frak{g}$ for $SHO$. Sometime  we also write
$\frak{g}(\underline{t})$ for $\frak{g}(m,m;\underline{t})$  for $t\in \mathbb{N}^m$.
\end{convention}

We close this section by recalling the following general notion and basic facts.
 Suppose  $X$ is a finite dimensional $\mathbb{Z}$-graded Lie superalgebra, $X=\oplus_{i\in \mathbb{Z}}X_i.$ By
$$\mathrm{Der}X:=\mathrm{Der}_{\bar{0}}X\oplus \mathrm{Der}_{\bar{1}}X$$
denote the derivation algebra of $X$, which is also a  $\mathbb{Z}$-graded Lie superalgebra,
$$\mathrm{Der}X=\sum_{i\in \mathbb{Z}}\mathrm{Der}_iX$$
where
$$\mathrm{Der}_iX=\{\phi\in \mathrm{Der}X\mid \phi(X_t)\subset X_{t+i}, \, \forall t\in \mathbb{Z}\}.$$
As in the usual, write
 $$\mathrm{Der}^-X:=\mathrm{span}_{\mathbb{F}}\{\phi\in \mathrm{Der}_iX\mid i<0\},\quad
 \mathrm{Der}^+X:=\mathrm{span}_{\mathbb{F}}\{\phi\in \mathrm{Der}_iX\mid i\geq0\},$$
  called the negative and nonnegative parts of the derivation algebra of $X$, respectively.
The element in $\mathrm{Der}^-X$ is called negative degree derivation and the element in $\mathrm{Der}^+X$
is called nonnegative degree derivation.
\section{Restrictedness and negative derivations}

As mentioned in the introduction our main purpose
is to determine the derivations of  the special odd Hamiltonian superalgebras.
Motivated by the method used in the
modular Lie algebra theory \cite[Lemma 6.1.3 and Theorem 7.1.2]{st}, in this paper we do not compute directly the derivations of the
  non-restricted and non-simple special odd Hamiltonian superalgebras but determine firstly the derivations
 (especially, those of negative degree) of the restricted and simple  special odd Hamiltonian superalgebras.
 From \cite{lh} the Lie superalgebra
   $\frak{g}^{(2)}$ is simple. Since
  we need the restrictednees of the Lie superalgebras under considerations in the process of determining derivations,
   in this section we first
   show that $\frak{g}^{(2)}(\underline{t})$ is restricted if and only if $\underline{t}=\underline{1}.$
Since a derivation is determined by its action on a generating set, we next give
 a generating set of the  restricted Lie superalgebra  $\frak{g}^{(2)}(\underline{1})$.
Finally, we
determine the derivations of negative $\mathbb{Z}$-degree for $\frak{g}^{(2)}(m,m;\underline{1})$,
since it is enough for determining the derivations in the general case in the subsequent sections.

Let us  introduce   some symbols for later use:
\begin{eqnarray*}
&&\widetilde{\mathcal{G}}:=\bigg\{\mathrm{T_H}\bigg(x^{(\alpha)}x^{u}-\sum _{i\in
\mathbf{I}(\alpha,u)}\Gamma_{i}^{q}(x^{(\alpha)}x^{u})\bigg)\bigg | x^{(\alpha)}x^{u}\in \mathcal{D}^{\ast}, q\in \widetilde{\mathbf{I}}(\alpha,u)\bigg\};\\
&&\mathfrak{A}_{1}:=\{\mathrm{T_H}(x^{(\alpha)}x^{u})\mid
\mathbf{I}(\alpha,u)=\widetilde{\mathbf{I}}(\alpha,u)=\emptyset\};\\
&&\mathfrak{A}_{2}:=\{\mathrm{T_H}(x^{(\alpha)}x^{u})\mid \mathbf{I}(\alpha,u)=\emptyset,
\widetilde{\mathbf{I}}(\alpha,u)\neq\emptyset\}.
\end{eqnarray*}
 We also write down some facts in \cite{lh}:
\begin{itemize}
\item[$\mathrm{(i)}$] \cite[Lemma 2.2]{lh}
 For $f\in \mathcal{O}(m,m;\underline{t}),$ $\mathrm{T_H}(f)\in \frak{g}$
 if and only if $\Delta(f)=0$.
\item[$\mathrm{(ii)}$] \cite[Theorem 2.7]{lh} $\frak{g}$ is spanned by
$\widetilde{\mathcal{G}}\bigcup\mathfrak{A}_{1}\bigcup\mathfrak{A}_{2}$
and $\frak{g}$ is a ${\mathbb Z}$-graded
subalgebra of $ W(m,m;\underline{t})$,
 $\frak{g}=\oplus_{i=-1}^{\xi-4}\frak{g}_i.$
 Then $\frak{g}$
is spanned by the elements of the form
\begin{equation}\label{bl1130}
\mathrm{T_H}\bigg(x^{(\alpha)}x^{u}-\sum _{i\in
\mathbf{I}(\alpha,u)}\Gamma_{i}^{q}(x^{(\alpha)}x^{u})\bigg)\quad q\in \widetilde{I}(\alpha, u).
\end{equation}
 For convenience, we call $x^{(\alpha)}x^{u}$ a leader of the element (\ref{bl1130}).\\

\item[$\mathrm{(iii)}$] \cite[Corollary 3.5]{lh} $\frak{g}^{(1)}$ is spanned by
 $\widetilde{\mathcal{G}}\bigcup\mathfrak{A}_{2}$
and $\frak{g}^{(1)}$ is a ${\mathbb Z}$-graded
subalgebra of $ W(m,m;\underline{t})$,
 $\frak{g}^{(1)}=\oplus_{i=-1}^{\xi-4}(\frak{g}^{(1)})_i.$
Moreover,
$$(\frak{g}^{(1)})_i=[(\frak{g}^{(1)})_{-1},(\frak{g}^{(1)})_{i+1}],\quad -1\leq i\leq\xi-5.$$
$$( \frak{g}^{(1)})_{\xi-4}=\mathrm{span}_{\mathbb{F}}\bigg\{\mathrm{T_H}\bigg(x^{(\pi-\varepsilon_i)}x^{\omega-\langle i' \rangle}-
\sum_{r\in\mathbf{Y_0}\backslash \{i\}}\Gamma^i_r
\big(x^{(\pi-\varepsilon_i)}x^{\omega-\langle i'
\rangle}\big)\bigg)\bigg|i\in \mathbf{Y_0}\bigg\}.$$ Put
$\mathcal{D}= \mathcal{D}_2\cup \mathcal{D}^{\ast}.$ By
(\ref{bl1130}),    $\frak{g}^{(1)}$ is spanned by the  elements of the form
\begin{equation}\label{eqbl1504}
\mathrm{T_H}\bigg(x^{(\alpha)}x^{u}-\sum _{i\in
\mathbf{I}(\alpha,u)}\Gamma_{i}^{q}(x^{(\alpha)}x^{u})\bigg)\quad q\in \widetilde{I}(\alpha, u),\quad x^{(\alpha)}x^{u} \in\mathcal{D}.
\end{equation}

\item[$\mathrm{(iv)}$]  \cite[Theorem 3.8]{lh}  $\frak{g}^{(2)}$ is a ${\mathbb Z}$-graded
subalgebras of $ W(m,m;\underline{t})$,
 $\frak{g}^{(2)}=\bigoplus_{i=-1}^{\xi-5}(\frak{g}^{(2)})_i.$
Moreover,
\begin{eqnarray*}
&&(\frak{g}^{(2)})_{i}=(\frak{g}^{(1)})_{i}, \quad -1\leq i\leq \xi-5;\\
&&(\frak{g}^{(2)})_{i-1}=[(\frak{g}^{(2)})_{-1},  (\frak{g}^{(2)})_{i}],\quad 0\leq i\leq \xi-5.
\end{eqnarray*}
\end{itemize}

\begin{theorem}\label{t2.2}
$\frak{g}^{(2)}(\underline{t})$ is  restricted
  if and only if $\underline{t}=\underline{1}.$
\end{theorem}
\begin{proof} Suppose  $\underline{t}=\underline{1}.$
Note that $W(m,m;\underline{1})$ is the full derivation algebra of the underlying algebra
$\mathcal{O}(m,m;\underline{1})$. One sees that  $W(m,m;\underline{1})$ is a restricted Lie superalgebra
with respect to the usual $p$-power  (mapping) and that the $p$-power fulfills that $(x_i\partial_i)^{p}=x_i\partial_i$
for $i \in \mathbf{Y}$ and vanishes on the other even standard basis elements, as in the Lie algebra case.
Thus it is sufficient to show that the even part of
$\frak{g}^{(2)}(\underline{1})$ is closed under the $p$-power. Note that the even part of $\frak{g}^{(2)}(1)$
is spanned by the elements of the form  (\ref{eqbl1504})
$$A:=\mathrm{T_H}\bigg(x^{(\alpha)}x^{u}-\sum _{i\in
\mathbf{I}(\alpha,u)}\Gamma_{i}^{q}(x^{(\alpha)}x^{u})\bigg)\quad q\in \widetilde{\mathbf{I}}(\alpha,u),\quad  x^{(\alpha)}x^{u} \in\mathcal{D},$$
where $|u|$ is odd. It is sufficient to show that $A^{p}\in  \frak{g}^{(2)}(\underline{1}).$
We shall frequently use the formula below without notice:
\[\mathrm{T_H}(x^{\alpha}x^u)^p= \left\{\begin{array}{ll}
\mathrm{T_H}(x^{\alpha}x^u), \, \,&\mbox{if}\, \, x^{\alpha}x^u=x_ix_{i'}, \,\mbox{for}\, \,  i\in\mathbf{Y_0},\\
0, &\mbox{otherwise},
\end{array}\right.\]
which is a direct consequence of \cite[Proposition 5.1]{lz1}.

If $x^{(\alpha)}x^u= x_ix_{i'}$ for $i\in\mathbf{Y_0}$ then
$$A^{p}=\big(\mathrm{T_H}(x_ix_{i'}-\Gamma_i^qx_ix_{i'})\big)^p
=\mathrm{T_H}(x_ix_{i'}-\Gamma_i^qx_ix_{i'})\in \frak{g}^{(2)}(1),\,q \in \mathbf{Y_0}, \, q\neq i.$$
Now  assume that $x^{(\alpha)}x^u\not= x_ix_{i'}$ for $i\in\mathbf{Y_0}$ and let us show that $A^{p}=0.$
\\

\noindent\textbf{Case 1:}
$|u|>1$, that is, $|u|\geq3$. Then $x^ux^{u-\langle i'\rangle-\langle j'\rangle}=x^{u-\langle i'\rangle}x^{u-\langle j'\rangle}=0.$
It follows that $[\mathrm{T_H}(x^{(\alpha)}x^u), \mathrm{T_H}(\Gamma^q_rx^{(\alpha)}x^u)]=0$ and hence $A^p=0.$
\\

\noindent\textbf{Case 2:} $|u|=1.$ Suppose  $u=\{ i'\}$, $i\in \mathbf{Y_0}$.
Since $A$ is a derivation of $\mathcal{O}(m,m;\underline{1})$,
it suffices to show that $A^p(x_j)=0$ for all $j\in \mathbf{Y}$.  We consider the following two subcases:\\

\noindent\textbf{Subcase 2.1:} $j\in  \mathbf{Y_0}$. We have
\begin{eqnarray*}
&&A(x_j)=0 \quad \mbox{for}\, \, j\in\mathbf{Y_0}\backslash\{i, q\};\\
&&A^p(x_j)=ax^{(p\alpha-(p-1)\varepsilon_j)}=0 \quad \mbox{for} \, \,j=i;\\
&&A^p(x_j)=bx^{(p\alpha-p\varepsilon_j+\varepsilon_q)}=0 \quad \mbox{for} \,\, j=q,
\end{eqnarray*}
where $a, b \in\mathbb{F}.$\\

\noindent\textbf{Subcase 2.2:} $j\in  \mathbf{Y_1}$. We have
\begin{eqnarray}\label{05151}
A^p(x_{j})=cx^{(p\alpha-(p-1)\varepsilon_i-\varepsilon_{j'})}x_{i'}
-dx^{(p\alpha-p\varepsilon_i-\varepsilon_{j'}+\varepsilon_q)}x_{q'},
\end{eqnarray}
where $c, d \in\mathbb{F}.$
In particular,
\begin{eqnarray}\label{05152}
\bigg(\mathrm{T_H}(x^{(\varepsilon_i+\varepsilon_{j'})}x_{i'})-x^{(2\varepsilon_{j'})}x_{j}\bigg)^3(x_{j})=0
\quad i\neq j'.
\end{eqnarray}
The equations (\ref{05151}) and (\ref{05152}) show $A^p(x_{j})=0$ unless
$\alpha=\varepsilon_i+\varepsilon_{j'}$ with distinct $i$, $j'$ and $q$.
Note that
$$[\mathrm{T_H}(x^{(\varepsilon_i+\varepsilon_{j'})}x_{i'}),
\mathrm{T_H}(x^{(\varepsilon_{j'}+\varepsilon_q)}x_{q'}]=0 \;
 \, \mbox{for distinct } i, j', \, q\in \mathbf{Y_0},$$
which implies that
$$\mathrm{T_H}\big(x^{(\varepsilon_i+\varepsilon_{j'})}x_{i'}
-x^{(\varepsilon_{j'}+\varepsilon_q)}x_{q'}\big)^p=0.$$
In conclusion, $A^{p}\in  \frak{g}^{(2)}(\underline{1})$ and hence $\frak{g}^{(2)}(\underline{1})$ is a restricted Lie superalgebra.

Suppose conversely  that $\frak{g}^{(2)}(\underline{t})$ is a restricted
Lie superalgebra. Then for every $i\in \mathbf{Y_0}$, $(\mathrm{ad}\partial_i)^p$
is   an inner derivation and
$(\mathrm{ad}\partial_i)^p$ is of $\mathbb{Z}$-degree $\geq -1$.
On the other hand we have $(\mathrm{ad}\partial_i)^p\in \mathrm{Der}_{-p}(\frak{g}^{(2)}(\underline{1}))$.
Consequently, $(\mathrm{ad}\partial_i)^p=0$ for all $i\in \mathbf{Y_0}$ which
 forces
$\underline{t}=\underline{1}.$ The proof is complete.
\end{proof}

The following lemma is simple but useful, the proof is similar to the one of the Lie algebra \cite[Proposition 3.3.5]{sf}.

\begin{lemma}\label{l2.3}
Let $L=\oplus_{i=-r}^{s}L_i$ be a simple, finite dimensional, and $\mathbb{Z}$-graded
Lie superalgebra. Then the following statements hold:
\item[$\mathrm{(1)}$]$L_{-r}$ and $L_s$ are irreducible $L_0$-modules.

\item[$\mathrm{(2)}$] $[L_0, L_s]=L_s,$ \quad $[L_0, L_{-r}]=L_{-r}.$

\item[$\mathrm{(3)}$] $C_{L_{s-1}}(L_1)=0,$ \quad $[L_{s-1}, L_1]=L_s.$

\item[$\mathrm{(4)}$] $C_L{(L^+)}=L_s,$ \quad $C_L{(L^-)}=L_{-r}.$
\end{lemma}

\begin{remark}\label{l2.5}
Let $T:=\mathrm{span}_{\mathbb{F}}\big\{T_{ij}=\mathrm{T_H}(x_ix_{i'}-x_jx_{j'}) \mid
i, j \in \mathbf{Y_0} , \, i\neq j\big\}.$ Obviously, $T$ is Abelian. From the proof of
Theorem \ref{t2.2} we know $(T_{ij})^p=T_{ij}$ which shows $T_{ij}$ is a toral. Consequently,
$T$ is a torus of $\frak{g}$.
In particular, $T$ is a torus of the restricted Lie superalgebra of $\frak{g}^{(2)}(\underline{1})$.
A direct computation shows that
\begin{eqnarray}
[T_{ij}, \mathrm{T_H}\big(x^{(\alpha)}x^u\big)]=(\delta_{i'\in u}-
\delta_{j'\in u}-\alpha_i+\alpha_j)\mathrm{T_H}\big(x^{(\alpha)}x^u\big).\label{05153}
\end{eqnarray}
Furthermore,
\begin{eqnarray*}
&&[T_{ij}, \mathrm{T_H}\bigg(x^{(\alpha)}x^{u}-\sum _{i\in
\mathbf{I}(\alpha,u)}\Gamma_{i}^{q}(x^{(\alpha)}x^{u})\bigg)]\\\nonumber
&=&(\delta_{i'\in u}-
\delta_{j'\in u}-\alpha_i+\alpha_j)\mathrm{T_H}\bigg(x^{(\alpha)}x^{u}-\sum _{i\in
\mathbf{I}(\alpha,u)}\Gamma_{i}^{q}(x^{(\alpha)}x^{u})\bigg),
\end{eqnarray*}
for $q\in \mathbf{Y_0}$.
\end{remark}

\begin{lemma}\label{l2.6}
Let $L=\oplus_{i=-r}^{s}L_i$ be a simple, finite-dimensional, and $\mathbb{Z}$-graded
Lie superalgebra. Let $M\subset L$ be a subalgebra that contains $L_{-1}\oplus L_1.$
If $M\cap L_{s-1}\neq 0,$ then $M=L.$
\end{lemma}
\begin{proof}
This is a direct consequence of  Lemma \ref{l2.3}.
\end{proof}

\begin{lemma}\label{l2.7}
$\frak{g}^{(2)}(\underline{1})$ is generated by
$\frak{g}^{(2)}(\underline{1})_{-1}\oplus \frak{g}^{(2)}(\underline{1})_1$.
\end{lemma}
\begin{proof}
Recall that $\frak{g}^{(2)}(\underline{1})=\frak{g}^{(2)}(m,m;\underline{1})$ is a graded subalgebra of $W(\underline{1})$.
Let $M$ denote the subalgebra generated
by $\frak{g}^{(2)}(\underline{1})_{-1}\oplus \frak{g}^{(2)}(\underline{1})_1$.
We proceed by induction on $m$.

Suppose  $m=3.$ Assume that $\frak{g}^{(2)}(3,3;\underline{1})_i\subset M$
for some $i \in \overline{1, 3p-6}$, and let $A$ be an element of
 $\frak{g}^{(2)}(3,3;\underline{1})_{i+1}$ with a leader $x^{(\alpha)}x^u$ (cf (\ref{eqbl1504})), that is
\begin{equation*}
A=\mathrm{T_H}\bigg(x^{(\alpha)}x^{u}-\sum _{i\in
\mathbf{I}(\alpha,u)}\Gamma_{i}^{q}(x^{(\alpha)}x^{u})\bigg)\quad q\in \widetilde{I}(\alpha, u),\quad x^{(\alpha)}x^{u} \in\mathcal{D}.
\end{equation*}
 Note that
$\widetilde{\mathbf{I}}(\alpha,u)\neq\emptyset$.
It is clear that $4\leq|\alpha|+|u|\leq 3p-3.$
Let us show $A \in M.$
We only have to consider  the following cases:\\

\noindent\textbf{Case 1:} $|u|=0$ and $|\alpha|\geq 4.$ One may assume without loss of the generality that $\alpha_1\geq 2.$ Then
\begin{eqnarray*}
 \bigg(\frac{\alpha_1(\alpha_1-1)}{2}-\alpha_1\alpha_2\bigg)
\mathrm{T_H}(x^{(\alpha)})
=\big[\mathrm{T_H}\big(x^{(\alpha-\varepsilon_1)}\big), \mathrm{T_H}\big(x^{(2\varepsilon_1)}
x_{1'}-x^{(\varepsilon_1+\varepsilon_2)}x_{2'}\big)\big].
\end{eqnarray*}

\item[$\mathrm{(i)}$] If $\alpha_2=0$, the equation shows that $\mathrm{T_H}(x^{(\alpha)}) \in  M.$

\item[$\mathrm{(ii)}$] If $\alpha_2\neq 0$, we have
\begin{eqnarray*}
 \bigg(\frac{\alpha_2(\alpha_2-1)}{2}-\alpha_1\alpha_2\bigg)
\mathrm{T_H}(x^{(\alpha)})=
\big[\mathrm{T_H}\big(x^{(\alpha-\varepsilon_2)}\big),
\mathrm{T_H}\big(x^{(2\varepsilon_2)}x_{2'}-x^{(\varepsilon_1+\varepsilon_2)}x_{1'}\big)\big].
\end{eqnarray*}
If
$$\frac{\alpha_1(\alpha_1-1)}{2}-\alpha_1\alpha_2\equiv
\frac{\alpha_2(\alpha_2-1)}{2}-\alpha_1\alpha_2\equiv 0\pmod{p}$$
 we obtain
 $\alpha_1=\alpha_2=p-1.$
Then $\alpha_3<p-1.$ Since
\begin{eqnarray*}
(1+\alpha_3)\mathrm{T_H}(x^{(\alpha)})
=\big[\mathrm{T_H}\big(x^{(\alpha-\varepsilon_1)}\big), \mathrm{T_H}\big(x^{(2\varepsilon_1)}x_{1'}-x^{(\varepsilon_1+\varepsilon_3)}x_{3'}\big)\big],
\end{eqnarray*}
we have $\mathrm{T_H}(x^{(\alpha)})\in M$.\\

\noindent \textbf{Case 2:} $|u|=1, \, |\alpha|\geq 3.$
One may assume without loss of generality that $u=\{1'\}$.

\item[$\mathrm{(i)}$] Suppose  $\alpha_1=0$ and $\alpha_2\geq 2$. We have
\begin{eqnarray*}
&&\bigg(\frac{\alpha_2(\alpha_2-1)}{2}+\alpha_2\bigg)
\mathrm{T_H}(x^{(\alpha)}x_{1'})\\
&=&\big[\mathrm{T_H}\big(x^{(\alpha-\varepsilon_2)}x_{1'}\big), \mathrm{T_H}\big(x^{(2\varepsilon_2)}x_{2'}-x^{(\varepsilon_1+\varepsilon_2)}x_{1'}\big)\big],
\end{eqnarray*}
and then $\mathrm{T_H}(x^{(\alpha)}x_{1'})\in M$ if $\frac{\alpha_2(\alpha_2-1)}{2}+\alpha_2\not\equiv 0\pmod{p}.$
On the other hand, if
$\frac{\alpha_2(\alpha_2-1)}{2}+\alpha_2\equiv 0 \pmod{p},$
we obtain $\alpha_2=p-1$.
Then $\alpha_3<p-1.$ Since
\begin{eqnarray*}
 (1+\alpha_3)
\mathrm{T_H}(x^{(\alpha)}x_{1'})=\big[\mathrm{T_H}\big(x^{(\alpha-\varepsilon_2)}x_{1'}\big),
 \mathrm{T_H}\big(x^{(2\varepsilon_2)}x_{2'}-x^{(\varepsilon_2+\varepsilon_3)}x_{3'}\big)\big],
\end{eqnarray*}
we have $A\in M$.\\

\noindent\item[$\mathrm{(ii)}$] Suppose  $\alpha_1>0$. One can assume that $\alpha_2< p-1.$
From Case 2 (i) we have, when $\alpha_3<p-1$
\begin{eqnarray*}
&&\mathrm{T_H}\big(x^{(\alpha)}x_{1'}-\Gamma_1^2x^{(\alpha)}x_{1'}\big)\\
&=&-\big[\mathrm{T_H}\big(x^{(\alpha_1\varepsilon_1)}x_{2'}\big),
\mathrm{T_H}\big(x^{((\alpha_2+1)\varepsilon_2+\alpha_3\varepsilon_3)}x_{1'}
\big)\big]\in M.
\end{eqnarray*}
When $\alpha_3=p-1$
\begin{eqnarray*}
&&\mathrm{T_H}\big(x^{(\alpha)}x_{1'}-\Gamma_1^2x^{(\alpha)}x_{1'}\big)\\
&=&\big[\mathrm{T_H}\big(x^{(\alpha_1\varepsilon_1+\varepsilon_3)}x_{2'}\big),
\mathrm{T_H}\big(x^{((\alpha_2+1)\varepsilon_2+(\alpha_3-1)\varepsilon_3)}x_{1'}
\big)\big]\in M.
\end{eqnarray*}
Similarly, we can obtain $\mathrm{T_H}(x^{(\alpha)}x_{1'}-\Gamma_1^3x^{(\alpha)}x_{1'})\in M$,
when $\alpha_3<p-1$.\\

\noindent\textbf{Case 3:} $|u|=2, \, |\alpha|\geq2.$
One can assume that  $u=\{1', 2'\}$, $\alpha_3<p-1$.

\item[$\mathrm{(i)}$] Suppose  $\alpha_1=\alpha_2=0$. Applying  case 2 (i) we have
\begin{eqnarray*}
&&\mathrm{T_H}\big(x^{(\alpha)}x_{1'}x_{2'}\big)\\
&=&-\frac{1}{\alpha_3+1}\big[\mathrm{T_H}\big(x^{(\alpha)}x_{1'}\big),
\mathrm{T_H}\big(x^{(\varepsilon_3)}x_{2'}x_{3'}-\Gamma_3^1x^{(\varepsilon_3)}x_{2'}x_{3'}\big)\big]\in M.
\end{eqnarray*}

\noindent\item[$\mathrm{(ii)}$] Suppose  $\alpha_1> 0$,  $\alpha_1<p-1$ or $\alpha_2<p-1$.
From Case 2 (ii) and Case 3 (i)
we have
\begin{eqnarray*}
&&(\alpha_3+1)\mathrm{T_H}\bigg(x^{(\alpha)}x_{1'}x_{2'}-
\sum_{r\in \mathbf{I}(\alpha, \{1', 2'\})}\Gamma_r^3x^{(\alpha)}x_{1'}x_{2'}\bigg)\\
&=&\big[\mathrm{T_H}\big(x^{(\alpha_3\varepsilon_3)}x_{1'}x_{2'}\big), \mathrm{T_H}\big(x^{(\alpha-(\alpha_3-1)\varepsilon_3)}x_{3'}-
\Gamma_3^qx^{(\alpha-(\alpha_3-1)\varepsilon_3)}x_{3'}\big)\big]\in M,
\end{eqnarray*}
where $q=1$ or $2$ such that  $\alpha_q< p-1.$\\

\noindent\item[$\mathrm{(iii)}$]  Suppose  $\alpha_1=\alpha_2=p-1$, $\alpha_3<p-2$.
Applying Cases 2 and 3 (ii), we obtain
\begin{eqnarray*}
&&\mathrm{T_H}\bigg(x^{(\alpha)}x_{1'}x_{2'}-
\sum_{r\in \{1, 2\}}\Gamma_r^3x^{(\alpha)}x_{1'}x_{2'}\bigg)\\
&=&-\big[\mathrm{T_H}\big(x^{(p-2)\varepsilon_1}x_{2'}x_{3'}\big), 
\mathrm{T_H}\big(x^{(\varepsilon_1+(p-1)\varepsilon_2+(\alpha_3+1)\varepsilon_3)}x_{1'}-
x^{((p-1)\varepsilon_2+(\alpha_3+2)\varepsilon_3)}x_{3'}\big)\big]\\
&&\in M.
\end{eqnarray*}

Now suppose  $m>3.$ Let
$$\overline{L}:=\mathrm{span}_{\mathbb{F}}\{\overline{X}\},
$$
where
\begin{eqnarray*}
\overline{X}:=\left\{\mathrm{T_H}\big(x^{(\alpha)}x^u-\sum_{r\in \mathbf{I}(\alpha, u)}\Gamma_r^qx^{(\alpha)}x^u\big)
\left|\begin{array}{l}
\alpha=\alpha_1\varepsilon_1+\alpha_2\varepsilon_2+\alpha_3\varepsilon_3\\
u\subset \{1', 2', 3',\}\\
|\alpha|+|u|\leq 3p-3\\
x^{(\alpha)}x^u\in \mathcal{D},\quad q\in \widetilde{I}(\alpha, u)
\end{array}
\right.
\right\}
\subset\frak{g}^{(2)}(\underline{1}).
\end{eqnarray*}
Let
$${L'}:=\mathrm{span}_{\mathbb{F}}\{X'\},
$$
where
\begin{eqnarray*}
{X'}:=\left\{\mathrm{T_H}\big(x^{(\alpha)}x^u-\sum_{r\in \mathbf{I}(\alpha, u)}\Gamma_r^qx^{(\alpha)}x^u\big)
\left|\begin{array}{l}
\alpha_i=0, \, i\in \{1, 2, 3\}\\
u\subset \omega':=\omega-\{1', 2', 3'\}\\
|\alpha|+|u|\leq (m-3)p-3\\
x^{(\alpha)}x^u\in \mathcal{D},\quad q\in \widetilde{I}(\alpha, u)
\end{array}
\right.
\right\}
\subset\frak{g}^{(2)}(\underline{1}).
\end{eqnarray*}
Obviously,  $\overline{L}\cong \frak{g}^{(2)}(3,3;\underline{1})$ and
${L'}\cong \frak{g}^{(2)}(m-3,m-3;\underline{1})$.

Let $\overline{\mathcal{E}}:=(p-1)(\varepsilon_1+\varepsilon_2+\varepsilon_3)$,
$\mathcal{E'}:=(p-1)(\varepsilon_4+\cdots+\varepsilon_m)$,
 $\mathcal{E}:=\overline{\mathcal{E}}+\mathcal{E'}$.
The induction hypothesis are applied to these algebras yield:
\begin{eqnarray*}
&&\overline{\Omega}:=\mathrm{T_H}\big(x^{(\overline{\mathcal{E}}-\varepsilon_1)}x_{3'}
-x^{(\overline{\mathcal{E}}-\varepsilon_3)}x_{1'}\big)\in \overline{L}\subset M,\\
&&\Omega':=\mathrm{T_H}\bigg(x^{(\mathcal{E'}-\varepsilon_{m-1}-\varepsilon_m)}x^{\omega'-\langle{m-1}'\rangle}\\
&&\quad \quad-\sum_{r\in \mathcal{A}}
\Gamma_r^{m-1}x^{(\mathcal{E'}-\varepsilon_{m-1}-\varepsilon_m)}x^{\omega'-\langle{m-1}'\rangle}\bigg) \in L'\subset M,
\end{eqnarray*}
where $\mathcal{A}=\mathbf{I}(\mathcal{E'}-\varepsilon_{m-1}-\varepsilon_m, \omega'-\langle{m-1}'\rangle)$.
Noting that
\begin{eqnarray*}
\Omega&:=&\mathrm{T_H}
\big(x^{(\varepsilon_1+\varepsilon_2+\varepsilon_{m-1})}x_{2'}-
x^{(\varepsilon_1+\varepsilon_{m-1}+\varepsilon_m)}x_{m'}\big)\\
&=&\big[\mathrm{T_H}\big(x^{(\varepsilon_{m-1}+\varepsilon_m)}x_{2'}\big), \mathrm{T_H}\big(x^{(\varepsilon_1+\varepsilon_2)}x_{m'}\big)]\in M
\end{eqnarray*}
and
\begin{eqnarray*}
B:=[\Omega, \overline{\Omega}]=-\mathrm{T_H}\big(x^{(\overline{\mathcal{E}}+
\varepsilon_{m-1})}x_{3'}-
x^{(\overline{\mathcal{E}}-\varepsilon_3+\varepsilon_{m-1}+\varepsilon_m)}x_{m'}\big)\in M,
\end{eqnarray*}
 we have
\begin{eqnarray*}
C:=[\Omega', B]
&=&\mathrm{T_H}\bigg(x^{(\mathcal{E}-\varepsilon_3
-\varepsilon_{m})}x^{\omega'-\langle{m-1}'\rangle}\\
&&\quad \quad-\sum_{r\in \mathcal{A}}
\Gamma_r^3x^{(\mathcal{E}-\varepsilon_3
-\varepsilon_{m})}x^{\omega'-\langle{m-1}'\rangle}\bigg)\\
&&\in M\cap \big(g^{(2)}(\underline{1})\big)_{mp-8}.
\end{eqnarray*}
Note that
$$\mathbf{I}(\mathcal{E}-\varepsilon_3
-\varepsilon_{m}, \omega'-\langle{m-1}'\rangle)=
\mathbf{I}(\mathcal{E'}-\varepsilon_{m-1}-\varepsilon_m, \omega-\langle{m-1}'\rangle).$$
Putting
\begin{eqnarray*}
D&:=&\mathrm{T_H}\big(x^{(\varepsilon_m)}x_{1'}x_{2'}x_{m-1'}\big)\\
&=&-\frac{1}{2}\big[\mathrm{T_H}\big(x^{(\varepsilon_m)}x_{1'}x_{2'}\big),
\mathrm{T_H}\big(x^{(\varepsilon_m)}x_{m-1'}x_{m'}+
x^{(\varepsilon_1)}x_{1'}x_{m-1'}\big)\big]\in M,
\end{eqnarray*}
we have
\begin{eqnarray*}
[D, C]&=&\mathrm{T_H}\bigg(x^{(\mathcal{E}-\varepsilon_3
-\varepsilon_{m})}x^{\omega-\langle 3'\rangle-\langle m'\rangle}\\
&&\quad\quad\quad\quad
-\sum_{r'\in \omega\backslash\{3',\,m'\}}\Gamma_r^3 x^{(\mathcal{E}-\varepsilon_3
-\varepsilon_{m})}x^{\omega-\langle 3'\rangle-\langle m'\rangle}\bigg)\\
&+&(-1)^{m-5}\mathrm{T_H}\bigg(x^{(\mathcal{E}-\varepsilon_1
-\varepsilon_{3})}x^{\omega-\langle 1'\rangle-\langle 3'\rangle}\\
&&\quad\quad\quad\quad
-\sum_{r'\in \omega\backslash\{1', \,3'\}}\Gamma_r^3 x^{(\mathcal{E}-\varepsilon_1
-\varepsilon_{3})}x^{\omega-\langle 1'\rangle-\langle 3'\rangle}\bigg)\\
&-&(-1)^{m-5}\mathrm{T_H}\bigg(x^{(\mathcal{E}-\varepsilon_2
-\varepsilon_{3})}x^{\omega-\langle 2'\rangle-\langle 3'\rangle}\\
&&\quad\quad\quad\quad
-\sum_{r'\in \omega\backslash\{2', \,3'\}}\Gamma_r^3 x^{(\mathcal{E}-\varepsilon_2
-\varepsilon_{3})}x^{\omega-\langle 2'\rangle-\langle 3'\rangle}\bigg)\\
&+&\mathrm{T_H}\bigg(x^{(\mathcal{E}-\varepsilon_3
-\varepsilon_{m-1})}x^{\omega-\langle 3'\rangle-\langle m-1'\rangle}\\
&&\quad\quad
-\sum_{r'\in \omega\backslash\{3', \,m-1'\}}\Gamma_r^3 x^{(\mathcal{E}-\varepsilon_3
-\varepsilon_{m-1})}x^{\omega-\langle 3'\rangle-\langle m-1'\rangle}\bigg)\\
&&\in M\cap \big(g^{(2)}(\underline{1})\big)_{mp-6}.
\end{eqnarray*}
Applying  Lemma \ref{l2.6}, we have
$\frak{g}^{(2)}(\underline{1})=M$, which is generated by
$\frak{g}^{(2)}(\underline{1})_{-1}\oplus \frak{g}^{(2)}(\underline{1})_1$.
\end{proof}
\begin{lemma}\label{l2.8}
Let $L=\oplus_{i=-r}^{s}L_i$ be a $\mathbb{Z}$-graded and centerless Lie superalgebra and $T\subset L_0\cap L_{\bar{0}}$ be an
Abelian subalgebra of $L$  such that $\mathrm{ad}x$ is semisimple for all $x \in T$.
If $\varphi \in \mathrm{Der}_{\mathbb{F}}(L)$ is homogeneous of
degree $t$, there is $e \in L_t$ such that $\big(\varphi-\mathrm{ad}e\big)\big|_T=0.$
\end{lemma}
\begin{proof}
The proof  is similar to the one of \cite[Proposition 8.4]{sf}.
\end{proof}
\begin{convention}\label{c2} 
Hereafter we suppose  $m>3$ for simplicity.
\end{convention}
\begin{theorem}\label{t2.9} 
$\mathrm{Der}^-(\frak{g}^{(2)}(\underline{1}))=\sum_{i\in \mathbf{Y}}\mathbb{F}\mathrm{ad}(\partial_i).$
\end{theorem}
\begin{proof}
Let $\varphi$ be a homogeneous derivation of degree $t<0$.
From Lemma \ref{l2.8} we may assume that $\varphi(\frak{g}^{(2)}(\underline{1})_{-1}+T)=0,$
where
$$T=\mathrm{span}_{\mathbb{F}}\{T_{ij}=\mathrm{T_H}(x_ix_{i'}-x_jx_{j'})\mid i, j \in \mathbf{Y_0}, i\neq j\}$$
is a torus of  $\frak{g}^{(2)}(\underline{1})$ (see Remark \ref{l2.5}). 
Since $\frak{g}^{(2)}(\underline{1})$ is generated by
$\frak{g}^{(2)}(\underline{1})_{-1}\oplus \frak{g}^{(2)}(\underline{1})_1$,
we may assume that $t\in \{-1, -2\},$ and only have to show $\varphi(\frak{g}^{(2)}(\underline{1})_1)=0$.
\\

\noindent\textbf{Case 1:} $t=-2.$ We can assert that $\varphi(\mathrm{T_H}(x^{(3\varepsilon_i)}))=0$, for any $i\in\mathbf{Y_0}$.

Assume that $$\varphi(\mathrm{T_H}(x^{(3\varepsilon_i)}))=\sum_{r\in \mathbf{Y}}a_r\partial_r.$$
Applying $\varphi$ to the equation 
$$[\mathrm{T_H}(x^{(3\varepsilon_i)}), \mathrm{T_H}(x^{(\varepsilon_k)}x_{j'})]=0$$
where $i, j, k \in \mathbf{Y_0}$ are  distinct,
we can obtain $a_{j'}=a_k=0$, then
$$\varphi(\mathrm{T_H}(x^{(3\varepsilon_i)}))=a_i\partial_i+a_j\partial_j+a_{i'}\partial_{i'}+a_{k'}\partial_{k'}.$$
Applying $\varphi$ to the equation
$$[T_{ij}, \mathrm{T_H}(x^{(3\varepsilon_i)})]=-3\mathrm{T_H}(x^{(3\varepsilon_i)}),$$
we obtain $a_i=a_j=a_{i'}=a_{k'}=0.$
Hence $\varphi(\mathrm{T_H}x^{(3\varepsilon_i)})=0$.

From a direct and simple computation we can obtain that $\frak{g}^{(2)}(\underline{1})_1$ is generated by 
$$\frak{g}^{(2)}(\underline{1})_0\oplus\sum_{i\in \mathbf{Y_0}}\mathbb{F}\mathrm{T_H}(x^{(3\varepsilon_i)}).$$
Hence $\varphi=0.$

\noindent\textbf{Case 2:} $t=-1$. We can assert that $\varphi(\mathrm{T_H}(x^{(2\varepsilon_i)}))=0$ and  $\varphi(\mathrm{T_H}(x_{i'}x_{j'}))=0$
$i, j\in \mathbf{Y_0}$, $i\neq j.$

Assume that $$\varphi(\mathrm{T_H}(x^{(2\varepsilon_i)}))=\sum_{r\in \mathbf{Y}}b_r\partial_r.$$
Applying $\varphi$ to the equation
$$[\mathrm{T_H}(x^{(2\varepsilon_i)}), T_{jk}]=0,$$
where $i, j, k \in \mathbf{Y_0}$ are  distinct,
we can obtain $b_{j}=b_{j'}=b_{k}=b_{k'}=0$, then
$$\varphi(\mathrm{T_H}(x^{(2\varepsilon_i)}))=b_i\partial_i+b_{i'}\partial_{i'}.$$
Applying $\varphi$ to the equation
$$[T_{ik}, \mathrm{T_H}(x^{(2\varepsilon_i)})]=-2\mathrm{T_H}(x^{(2\varepsilon_i)}),$$
we obtain $b_i=b_{i'}=0.$
Hence $\varphi(\mathrm{T_H}x^{(2\varepsilon_i)})=0$.

Assume that $$\varphi(\mathrm{T_H}(x_{i'}x_{j'}))=\sum_{r\in \mathbf{Y}}c_r\partial_r.$$
Applying $\varphi$ to the equation
$$[\mathrm{T_H}(x_{i'}x_{j'}), T_{ij}]=0,$$
we can obtain $c_{j}=c_{j'}=c_{i}=c_{i'}=0$, hence
$$\varphi(\mathrm{T_H}(x_{i'}x_{j'}))=\sum_{r\in \mathbf{Y}\backslash\{i, i', j, j'\}}c_r\partial_r.$$
For $m\geq 4$, we can put $k, l \in \mathbf{Y_0}$ satisfying $k\neq l$, $k,l \neq i,j$. Applying $\varphi$ to the equation
$$[T_{kl}, \mathrm{T_H}(x_{i'}x_{j'})]=0,$$ 
we obtain $b_i=b_{i'}=0.$
Hence $\varphi(\mathrm{T_H}(x_{i'}x_{j'}))=0$.
Note that 
$\frak{g}^{(2)}(\underline{1})_0$ is generated by
$$T\oplus\sum_{i, j\in \mathbf{Y_0}, i\neq j}\mathbb{F}\mathrm{T_H}(x^{(2\varepsilon_i)})\oplus\sum_{i, j\in \mathbf{Y_0}, i\neq j}\mathbb{F}\mathrm{T_H}(x_{i'}x_{j'}).$$
Hence $\varphi(\frak{g}^{(2)}(\underline{1})_0)=0$.
 Consequently,
$$[\frak{g}^{(2)}(\underline{1})_{-1}, \varphi\big(\frak{g}^{(2)}(\underline{1})_1\big)]=
\varphi\big([\frak{g}^{(2)}(\underline{1}\big)_{-1}, \frak{g}^{(2)}(\underline{1})_1]\big)=0.$$
By means of the transitiveness of the simple algebra, we have
$\varphi\big(\frak{g}^{(2)}(\underline{1})_1\big)=0.$
Hence $\varphi=0$. From Lemma \ref{l2.8} the conclusion holds.
\end{proof}

%%%%%%%%%%%%%%%%%%%%%%%%%%%%%%%%%%%%%%%%%%%%%%%%%%%%%%%%%%%%%%%%%%%%%%%%%%%%%%%%%%%%%%%%%%%%%%%%%%%%%%%%%%%%%%%%%%%%
%%%%%%%%%%%%%%%%%%%%%%%%%%%%%%%%%%%%%%%%%%%%%%%%%%%%%%%%%%%%%%%%%%%%%%%%%%%%%%%%%%%%%%%%%%%%%%%%%%%%%%%%%%%%%%%%
%%%%%%%%%%%%%%%%%%%%%%%%%%%%%%%%%%%%%%%%%%%%%%%%%%%%%%%%%%%%%%%%%%%%%%%%%%%%%%%%%%%%%%%%%%%%%%%%%%%%%%%%%%%%%
\section{Derivations}
In this section, we will determine derivations of
$\frak{g}$, $\frak{g}^{(1)}$ and $\frak{g}^{(2)}$.
Firstly, we study  the derivations  of negative $\mathbb{Z}$-degree for
$\frak{g}$, $\frak{g}^{(1)}$ and $\frak{g}^{(2)}$, by virtue of
the same subjects of the restricted Lie superalgebra $\frak{g}^{(1)}(\underline{1})$.
Secondly, we discuss the normalizers  of  $\frak{g}$, $\frak{g}^{(1)}$ and $\frak{g}^{(2)}$
 in $W.$ Finally, we obtain the
derivations of $\frak{g}$, $\frak{g}^{(1)}$ and $\frak{g}^{(2)}$.
\begin{lemma}\label{l3.1}
Let $M$ denote a subalgebra of $\frak{g}(\underline{s})$, $\underline{1}\leq\underline{k}\leq\underline{s}$.
 If
$$\frak{g}^{(2)}(\underline{k})+\mathbb{F}\mathrm{T_H}\big(x^{(p^{k_i}+1)\varepsilon_i}\big)\subset M$$
for some $i$, then $\frak{g}^{(2)}(\underline{k}+\varepsilon_i)\subset M.$
\end{lemma}
\begin{proof}
Observing that under consideration $\frak{g}^{(2)}(\underline{k}+\varepsilon_i)$
is a simple $\mathbb{Z}$-graded subalgebra of $\frak{g}(\underline{s}),$ that is
$$g^{(2)}(\underline{k}+\varepsilon_i)=\bigoplus_{r=-1}^n(g^{(2)}(\underline{k}+\varepsilon_i))_r,$$
where $n=\Sigma_{j=1}^mp^{k_j}-4.$ From Lemma \ref{l2.3}
we only have to prove that
$$M\cap \frak{g}^{(2)}(\underline{k}+\varepsilon_i)_{n-1}\neq(0).$$

In order to accomplish this,  we observe that:

for $1\leq a$, $1\leq b\leq p-1$, $1\leq r\leq p$
\[
\begin{pmatrix}
rp^a-b\\
p^a
\end{pmatrix}
\equiv
\begin{pmatrix}
r-1\\
1
\end{pmatrix}
=r-1\pmod{p}
\]
holds.
Assume inductively and without loss of the generality that $M$ contains
\begin{eqnarray*}
E_r&:=&\mathrm{T_H}\bigg(x^{(\mathcal{E}_k-\varepsilon_1-\varepsilon_2+(r-1)p^{k_i}\varepsilon_i)}
x^{\omega-\langle 1'\rangle}\\
&&\quad\quad-\sum_{q\in \mathcal{B}_1}\Gamma_q^1x^{(\mathcal{E}-\varepsilon_1-\varepsilon_2+(r-1)p^{k_i}\varepsilon_i)}
x^{\omega-\langle 1'\rangle}\bigg),
\end{eqnarray*}
where  $1\leq r\leq p-1$, $i\neq 1, 2$ and
\begin{eqnarray*}
\mathcal{E}_k&=&(p^{k_1}-1)\varepsilon_1+(p^{k_2}-1)\varepsilon_2+\cdots+(p^{k_m}-1)\varepsilon_m;\\
\mathcal{B}_1&=&\mathbf{I}(\mathcal{E}_k-\varepsilon_1-\varepsilon_2+(r-1)p^{k_i}\varepsilon_i,
\omega-\langle 1'\rangle).
\end{eqnarray*}
Then, we obtain
\begin{eqnarray*}
&&[\mathrm{T_H}(x^{(p^{k_i}+1)\varepsilon_i}), E_r]\\
&=&r(-1)^{i-2}\mathrm{T_H}\bigg(x^{(\mathcal{E}_k-\varepsilon_1-\varepsilon_2+rp^{k_i}\varepsilon_i)}
x^{\omega-\langle 1'\rangle-\langle i'\rangle}\\
&&-\sum_{q\in \mathcal{B}_2}\Gamma_q^1x^{(\mathcal{E}_k-\varepsilon_1-\varepsilon_2+rp^{k_i}\varepsilon_i)}
x^{\omega-\langle 1'\rangle-\langle i'\rangle}\bigg),
\end{eqnarray*}
where $\mathcal{B}_2=\mathbf{I}(\mathcal{E}_k-\varepsilon_1-\varepsilon_2+rp^{k_i}\varepsilon_i,
\omega-\langle 1'\rangle-\langle i'\rangle).$

By induction we obtain
\begin{eqnarray*}
&&\mathrm{T_H}\bigg(x^{(\mathcal{E}_k-\varepsilon_1-\varepsilon_2+(p-1)p^{k_i}\varepsilon_i)}
x^{\omega-\langle 1'\rangle-\langle i'\rangle}\\
&&-\sum_{q\in \mathcal{B}_2}\Gamma_q^1x^{(\mathcal{E}_k-\varepsilon_1-\varepsilon_2+(p-1)p^{k_i}\varepsilon_i)}
x^{\omega-\langle 1'\rangle-\langle i'\rangle}\bigg)\\
&&\in M\cap \frak{g}^{(2)}(\underline{k}+\varepsilon_i)_{n-1}.
\end{eqnarray*}

Hence the assertion holds.
\end{proof}

\begin{theorem}\label{t3.2}
Let $X$ be a $\mathbb{Z}$-graded subalgebra of $\frak{g}(m, m; \underline{t})$ containing  $\frak{g}^{(\infty)}$
and $\underline{s}$ be any element of $\mathbb{N}^m$ with $\underline{t}\leq \underline{s}$.
Then
$$\mathrm{Der}^-(X, \frak{g}(\underline{s}))=\mathrm{span}_{\mathbb{F}}\bigg\{
\{(\mathrm{ad}_X(\partial_i))^{p^{k_i}}\mid i\in\mathbf{Y_0}, \, 1\leq k_i< t_i\}\cup
\{\mathrm{ad}_X(\partial_i)\mid i\in\mathbf{Y}\}\bigg\}.$$
\end{theorem}
\begin{proof}
Let $T$ be the torus of $X_0$ mentioned in Remark \ref{l2.5}. Then
$$\mathrm{Der}^-(X, \frak{g}(\underline{s}))=\sum_{\mu\in T^*}\mathrm{Der}^-(X, \frak{g}(\underline{s}))_{\mu}$$
decomposes into the direct sum of $T$-weight spaces.
Take $d \in \mathrm{Der}^-(X, \frak{g}(\underline{s}))_{\mu}$ for some $\mu\neq 0,$
and $t \in T$ with $\mu(t)\neq 0.$
For arbitrary $u \in X$, we obtain
$$\mu(t)d(u)=(t\cdot d)(u)=[t, d(u)]-d[t, u]=-[d(t), u].$$
Hence $d=\mathrm{ad}_X(-\mu(t)^{-1}d(t))\in \mathrm{ad}_X\frak{g}(\underline{s}).$
According to $t \in T\subset X_0\cap X_{\bar{0}}$, we have
$d\in \mathrm{span}_{\mathbb{F}}\{\mathrm{ad}_X(\partial_i)\mid i\in\mathbf{Y}\}$,
thus we only have to determine homogeneous derivations $d$ from
$X$ to $\frak{g}(\underline{s})$ of degree $t<0$ vanishing on given torus $T$ of
$X_0$.
For $\frak{g}$ we have $d(X_{-1})\subset \frak{g}(\underline{s})_{-1+t}=0,$
hence $d(X\cap\ker(\mathrm{ad}\partial_i^p))\subset X\cap\ker(\mathrm{ad}\partial_i^p)$
and therefore $d$ maps $X\cap \frak{g}(\underline{1})$ into $X\cap \frak{g}(\underline{1})$.
Thus $d$ defines by restriction a derivation of $\frak{g}^{(2)}(\underline{1})$.
Applying Theorem \ref{t2.9} we obtain that $d-\sum_{i\in Y}\alpha_i\mathrm{ad}\partial_i$
vanishes on $\frak{g}^{(2)}(\underline{1})$ for a suitable choice of $\alpha_i \in \mathbb{F}$.
Thus we may assume that $\frak{g}^{(2)}(\underline{1})\subset \ker d.$

Take $\underline{1}\leq\underline{k}\leq\underline{t}$ to be maximal subject to the condition
$\frak{g}^{(2)}(\underline{k})\subset \ker d.$ Then
$$\big[\frak{g}^{(2)}(\underline{1}),
d\big(X\cap \overline{HO}(m,m;\underline{k})\cap \overline{S}(m,m;\underline{k})\big)\big]
\subset d(\frak{g}^{(2)}(\underline{k}))=0,$$
whence
$d(X\cap \overline{HO}(m,m;\underline{k})\cap \overline{S}(m,m;\underline{k}))
\subset\big\{D \in \frak{g}(\underline{s})\big|
 [\frak{g}^{(2)}(\underline{1}), D]=(0)\big\}=0$.
 This is the claim if $\underline{k}=\underline{t}$.
 Suppose  $\underline{k}<\underline{t}$, and let $i$ be an index for which
 $k_i<t_i$.
Take $E\in\frak{g}(\underline{s})$ as $E:=\mathrm{T_H}(x^{(p^{k_i}+1)\varepsilon_i})$.
Lemma \ref{l3.1} proves that $E \not\in\ker d.$
However, a computation shows that
$[E, \frak{g}(\underline{k})_{-1}]\subset X\cap \overline{HO}(m,m;\underline{k})\cap \overline{S}(m,m;\underline{k})\subset\ker d$,
whence
$$[d(E), \frak{g}(\underline{k})_{-1}]=0.$$
This means $d(E) \in \sum_{j \in \mathbf{Y}}\mathbb{F}\partial_j.$
We may assume that $d$ vanishes on the torus $T$.
Considering eigenvalues we obtain that there exists $\alpha \in \mathbb{F}$
such that
$$d(E)=\alpha\partial_{i'}, \, i\in\mathbf{Y_0}.$$
Thus $d'=d-\alpha\mathrm{ad}\partial_i^{p^{k_i}}$ annihilates
$\frak{g}^{(2)}(\underline{k})+\mathbb{F}E.$
Lemma \ref {l3.1} proves that $d'$ annihilates $\frak{g}^{(2)}(\underline{k}+\varepsilon_i).$
We now proceed by induction.

Thus we may assume that $\frak{g}^{(2)}(\underline{t}) \subset \ker d.$
As above we then conclude $d(X)=0.$
\end{proof}

\begin{remark}\label{r1}
We use the method for modular  Lie algebras \cite[Lemmas 5.2.6 and 6.1.3]{st} to prove
Lemma \ref{l3.1} and Theorem \ref{t3.2}.
\end{remark}

By virtue of  Theorem \ref{t3.2}, we can
determine the negative part of  the derivation algebra of
 $\frak{g}$, $\frak{g}^{(1)}$ and $\frak{g}^{(2)}$ as follows:
\begin{proposition}\label{t3.3} We have
\begin{eqnarray*}
\quad\mathrm{Der}^-(X)=
\mathrm{span}_{\mathbb{F}}\big(
\{(\mathrm{ad}_X(\partial_i))^{p^{k_i}}| i\in\mathbf{Y_0},\, 1\leq k_i< t_i\}\cup
\{\mathrm{ad}_X(\partial_i)| i\in\mathbf{Y}\}\big),
\end{eqnarray*}
where $X=\frak{g}$, $\frak{g}^{(1)}$, or $\frak{g}^{(2)}.$
\end{proposition}

Now we only have to investigate the nonnegative part of the derivation algebras.
To do that, let us first consider the normalizers of $\frak{g}$, $\frak{g}^{(1)}$, and $\frak{g}^{(2)}$.

\begin{lemma}\label{l3.5}
$\mathrm{Nor}_W(X)
\cap W_t\subseteq (\overline{HO}_t\cap\overline{S}_t)$ $t\in \mathbb{N}$,
where $X=\frak{g}$, $\frak{g}^{(1)}$, or $\frak{g}^{(2)}.$
\end{lemma}
\begin{proof}
Let $X:=\frak{g}$, $\frak{g}^{(1)}$, or $\frak{g}^{(2)}.$
Suppose
\[
E=\sum\limits_{j=1}^{2m}g_j\partial_j\in {\rm{Nor}}_W(X) \cap W_t,
\]
where $g_j\in\mathcal{O}( m,m;\underline{t}) _{t+1},$
$j\in \mathbf{Y}.$ Then there exists $f_i\in \mathcal{O}( m,m;\underline{t})_{t+1} $ satisfying
$\Delta (f_i)=0,$ such that
\[
[ \partial_i,E] =\mathrm{T_H}( f_i), \, i\in \mathbf{Y}.
\]
Note that
$$[ \partial_i,E] =\sum\limits_{j=1}^{2m}\partial_i
(g_j) \partial_j ,\quad \quad \mathrm{T_H}( f_i)
=\sum\limits_{j=1}^{2m}( -1) ^{\mu(j)\mathrm{p}(f_i)}\partial_j( f_i) \partial_{j'}. $$
We have
\begin{equation}\label{bl1310}
\partial_i( g_{j'}) =(-1) ^{\mu(j)\mathrm{p}(f_i)}\partial_j( f_i) ,\quad i\in \mathbf{Y}.\
\end{equation}
Observe that
\begin{eqnarray*}
\mathrm{T_H}(\partial_i(f_j)) &=&(-1)^{\mu(i)\mathrm{p}(f_j)}[\mathrm{T_H}(f_j) ,
\mathrm{T_H}( x_{i'})] \\
&=&(-1)^{\mu(i)\mathrm{p}(f_j)}\big[
\big[ \partial_j,E\big] ,(-1) ^{\mu(i')\mathrm{p}(x_{i'})}\partial_i\big] \\
&=&[\partial_i,[ \partial_j,E] ],
\end{eqnarray*}
where $i,$ $j \in \mathbf{Y}$.
Since $[ \partial_i, \partial_j] =0,$ we obtain the following equation:
\begin{equation}\label{bl1315}
\mathrm{T_H}(\partial_i( f_j)) =(-1) ^{\mu (i) \mu (j) }
\mathrm{T_H}(\partial_j( f_i)) ,\quad i,j\in \mathbf{Y}.
\end{equation}
Equations (\ref{bl1310}) and (\ref{bl1315}) yield
\begin{eqnarray*}
&&(-1)^{\mu(i)\mathrm{p}(f_j)}\partial_j (g_{i'}) -(-1)^{\mu(i)\mu(j)+\mu(j)\mathrm{p}(f_i)}\partial_i(g_{j'})\\
&=&\partial_i( f_j) - (-1) ^{\mu (i)
\mu ( j) }\partial_j( f_i) \in \ker(\mathrm{T_H}) =\mathbb{F}\cdot 1.
\end{eqnarray*}
Noting that $g_k\in \mathcal{O}( m,m;\underline{t}) _{t+1},\ k\in \mathbf{Y},$ we obtain that
\[
(-1)^{\mu(i)\mathrm{p}(f_j)}\partial_j
(g_{i'}) -(-1)^{\mu(i)\mu(j)+\mu(j)\mathrm{p}(f_i)}\partial_i
(g_{j'}) \in \mathbb{F}\cdot 1\cap \mathcal{O}( m,m;\underline{t})_{t}.
\]
The assumption that $t>0$ yields
\[
(-1) ^{\mu ( i)\mathrm{p}(f_j)}\partial_j
(g_{i'}) =( -1) ^{\mu ( i) \mu
( j) +\mu ( j) \mathrm{p}(f_i)}\partial_i
(g_{j'}) .
\]
Since $\mathrm{p}(f_i)=\mu( i) +\mathrm{p}(E)+\overline{1},$ it follows that
\[
\partial_i( g_{j'}) =( -1) ^{\mu
(i) \mu ( j) +( \mu( i) +\mu
( j))( \mathrm{p}(E)+\overline{1})}\partial_j( g_{i'}) .
\]
Hence $E\in \overline{HO}$.
Since
\[
[\partial_i,E]= \mathrm{T_H}( f_i) \in X, \ i\in \mathbf{Y},
\]
we obtain $\mathrm{div}[\partial_i,E]=0.$ By virtue of
(\ref{bl1305}) we have $E\in \overline{S}$.

Hence
$\mathrm{Nor}_W(X)
\cap W_t\subseteq (\overline{HO}_t\cap\overline{S}_t),$ $t\in \mathbb{N}$.
\end{proof}

\begin{lemma}\label{l3.6}
Let $H=\mathrm{span}_{\mathbb{F}}\{\mathrm{T_H}(x_ix_{i'})\mid i \in \mathbf{Y_0}\}$,
$h_1=\sum\nolimits_{i=1}^{m}x_{i'}\partial_{i'}.$ Then ${\mathrm{Nor}}_W(X) \cap W_0\subseteq \frak{g}_0
+H+\mathbb{F}\cdot h_1,$  where $X=\frak{g}$, $\frak{g}^{(1)}$, or $\frak{g}^{(2)}.$
\end{lemma}
\begin{proof}
Let $X:=\frak{g}$, $\frak{g}^{(1)}$, or $\frak{g}^{(2)}.$
Let $E$ be a $\mathbb{Z}_2$-homogeneous element of $\mathrm{Nor}_W(X) \cap
W_0$. Then
\[
E=\sum\limits_{i=1}^{2m}\sum\limits_{j=1}^{2m}\alpha
_{ij}x_i\partial_j,\quad \alpha _{ij}\in \mathbb{F}.
\]
Given $i\in \mathbf{Y_0},\ j\in \mathbf{Y_1}$ and $i\neq j'$, we have $\mathrm{T_H}( x_ix_j) \in \frak{g}_0$ and
\begin{eqnarray*}
[\mathrm{T_H}( x_ix_j) ,E]
&=&\Big( \alpha _{i'i'}x_j-\alpha _{j'i'}x_i-\sum\limits_{k=1}^{2m}\alpha _{kj}x_k\Big)
\partial_{i'}\\
&+&\Big( \alpha _{i'j'}x_j-\alpha_{j'j'}x_i+
\sum\limits_{k=1}^{2m}\alpha _{ki}x_k\Big) \partial_{j'}
+\sum_{k\neq i,j}( \alpha _{i'k'}x_j-\alpha _{j'k'}x_i)\partial_{k'}.
\end{eqnarray*}
Let $a_l$ denote the coefficient of  $\partial_l$ in the right side
of equation  above. Note that
$[\mathrm{T_H}( x_ix_j) ,E] \in \overline{HO}\cap\overline{S}.$
Since $\mathrm{p}([\mathrm{T_H}( x_ix_j) ,E]) =\mathrm{p}(E),$ by virtue of the equality
\[
\partial_i( a_{j'}) =( -1) ^{\mu(i) \mu( j) +( \mu( i) +\mu( j))(\mathrm{p}(E)+\overline{1})}
\partial_j( a_{i'}) ,
\]
an elementary computation shows that
\begin{equation}\label{bl05.063}
( -1) ^{\mathrm{p}(E)}\alpha _{ii}+\alpha _{i'i'}=\alpha _{jj}+( -1)
^{\mathrm{p}(E)} \alpha_{j'j'},\quad i\in \mathbf{Y_0},\ j\in
\mathbf{Y_1},\ i\neq j'.
\end{equation}
Similarly, by virtue of equations
\[
\partial_j( a_{k'}) =( -1) ^{\mu(k) \mu( j) +( \mu ( k) +\mu
( j))(\mathrm{p}(E)+\overline{1})}\partial_k( a_{j'})
\]
and
\[
\partial_i( a_{k'}) =( -1) ^{\mu(k) \mu( i) +( \mu ( k) +\mu
( i))(\mathrm{p}(E)+\overline{1})}\partial_k( a_{i'}),
\]
we obtain that
\begin{equation}\label{bl05.064}
\alpha _{ki}=-( -1) ^{\mu ( k')\mathrm{p}(E)}\alpha _{i'k'},\quad
i\in \mathbf{Y_0},\ k\in \mathbf{Y}\backslash\{ i\}
\end{equation}
and
\begin{equation}\label{bl05.065}
\alpha _{kj}=( -1) ^{\mu( k)(\mathrm{p}(E)+ \overline{1}) }\alpha _{j'k'},
\quad j\in \mathbf{Y_1},\ k\in \mathbf{Y}\backslash \{ j\} .
\end{equation}

\noindent\textbf{Case 1:} If $\mathrm{p}(E)=\overline{0},$ it is easy to see from (\ref{bl05.063})--(\ref{bl05.065}) that
\begin{equation}\label{bl05.066}
\alpha _{ii}+\alpha _{i'i'}=\alpha _{jj}+\alpha_{j'j'},
\quad i\in \mathbf{Y_0},\ j\in \mathbf{Y_1},
\end{equation}
\begin{equation}\label{bl05.067}
\alpha _{ki}=-\alpha _{i'k'},\quad i\in \mathbf{Y_0},\
k\in \mathbf{Y}\backslash \{ i\} ,
\end{equation}
\begin{equation}\label{bl05.068}
\alpha _{kj}=( -1) ^{\mu ( k) }\alpha_{j'k'},
\quad j\in\mathbf{ Y_1},\ k\in \mathbf{Y}\backslash\{ j\} .
\end{equation}
Note that $\alpha _{ij}=0,$ whenever $\mu( i) \neq \mu ( j) .$
We conclude from (\ref{bl05.067}) and (\ref{bl05.068}) that
\begin{equation}\label{bl05.069}
\alpha _{ij}=( -1) ^{\mu ( i) +\mu(j') }\alpha _{j'i'},\quad i, j\in
\mathbf{Y},\ i\neq j.
\end{equation}
We may suppose by (\ref{bl05.066}) that $\alpha _{ii}+\alpha _{i'i'}=\alpha$,
for any $i\in \mathbf{Y_0}.$ Applying (\ref{bl05.069}) we have
\begin{eqnarray*}
E &=&\sum\limits_{i=1}^{2m}\alpha _{ii}x_i\partial_i+\frac 12\sum_{i\neq
j}( -1) ^{\mu( j') }\alpha_{ij}
\mathrm{T_H}( x_ix_{j'}) .
\end{eqnarray*}
Moreover,
\begin{eqnarray*}
\sum\limits_{i=1}^{2m}\alpha _{ii}x_i\partial_i=-\sum\limits_{i=1}^{m}\alpha_{ii}
\mathrm{T_H}( x_ix_{i'})+\alpha h_1.
\end{eqnarray*}
Hence $E\in \frak{g}_0+H+\mathbb{F}\cdot h_1.$\\

\noindent\textbf{Case 2:} If $\mathrm{p}(E)=\overline{1},$ then $\alpha _{ij}=0$ whenever
$i,j\in \mathbf{Y}$ and  $\mu( i) =\mu ( j) .$ By virtue of (\ref{bl05.064}) and (\ref{bl05.065}),
we have
\begin{equation}\label{bl05.0610}
\alpha _{ki}=( -1) ^{\mu( k) }\alpha_{i'k'},
\quad i\in \mathbf{Y_0},\ k\in\mathbf{ Y}\backslash\{ i\}
\end{equation}
and
\begin{equation}\label{bl05.0611}
\alpha _{kj}=\alpha _{j'k'},\quad j\in \mathbf{Y_1},\
k\in \mathbf{Y}\backslash \{ j\} .
\end{equation}
Observe that (\ref{bl05.0610}) and (\ref{bl05.0611}) imply that $ \alpha _{kl}=( -1) ^{\mu(k) }
\alpha _{l'k'},\ k,l\in \mathbf{Y}. $ Therefore, we obtain that
\begin{eqnarray*}
E =\sum_{\mu( i) \neq \mu( j) }\alpha
_{ij}x_i\partial_j =\frac 12\sum_{\mu ( i) \neq \mu (j) }
( -1) ^{\mu( i) }\alpha _{ij}\mathrm{T_H}( x_ix_{j'}).
\end{eqnarray*}
Now we obtain the desired result.
\end{proof}

Note that $HO(m,m;\underline{t})$ is an ideal of $\overline{HO}(m,m;\underline{t})$
and $S'(m,m;\underline{t})$ is an ideal of $\overline{S}(m,m;\underline{t})$. From the definitions we have the following
\begin{proposition}\label{p1}
$X$ is an ideal of $\overline{HO}(m,m;\underline{t})\cap\overline{S}(m,m;\underline{t})$, where
$$X=\frak{g}(m,m;\underline{t}), \frak{g}(m,m;\underline{t})^{(1)}, \, \mbox{or}\, \, \frak{g}(m,m;\underline{t})^{(2)}.$$
\end{proposition} In conclusion, we can obtain:
\begin{theorem}\label{l3.7}
$\mathrm{Nor}_W(X)=\overline{HO}\cap\overline{S}\oplus\mathbb{F}\cdot h_1$,
where $X=\frak{g}$, $\frak{g}^{(1)}$, or $\frak{g}^{(2)}.$
\end{theorem}
\begin{proof}
Suppose  $X:=\frak{g}$, $\frak{g}^{(1)}$, or $\frak{g}^{(2)}.$
Applying
Lemmas \ref{l3.5} and \ref{l3.6} we have
$\mathrm{Nor}_W(X)\subset\overline{HO}\cap\overline{S}\oplus\mathbb{F}\cdot h_1$.
Note that $[h_1, \mathrm{T_H}(x^\alpha x^u)]=(|u|-1)\mathrm{T_H}(x^\alpha x^u)$, that is $h_1 \in \mathrm{Nor}_W(X).$  Hence
by virtue of Proposition \ref{p1} we obtain $\mathrm{Nor}_W(X)=\overline{HO}\cap\overline{S}\oplus\mathbb{F}\cdot h_1.$
\end{proof}
\begin{remark}\label{r3}
Let $X=\frak{g}$, $\frak{g}^{(1)}$ or $\frak{g}^{(2)}$.
Suppose  $h:=\sum_{i\in \mathbf{Y}}x_i\partial_i$. We obtain
$\mathrm{ad}h$ is the $\mathbb{Z}$-degree derivation of $X$, that is for all $\mathbb{Z}$ homogeneous element $A\in X_i,$
$$[h, A]=iA.$$
Note that $h=-\sum_{i\in \mathbf{Y_0}}\mathrm{T_H}(x_ix_{i'})+2h_1$,
where $\mathrm{T_H}(x_ix_{i'})\in \overline{HO}\cap \overline{S}.$
We can obtain $\mathrm{Nor}_W(X)=\overline{HO}\cap\overline{S}\oplus\mathbb{F}\cdot h.$
\end{remark}
Finally we characterize the derivations of $\frak{g}$, $\frak{g}^{(1)}$ and $\frak{g}^{(2)}$.
\begin{theorem}\label{t3.9} Suppose  $X=\frak{g}$, $\frak{g}^{(1)}$, or $\frak{g}^{(2)},$ we have
$$\mathrm{Der}(X)=\mathrm{ad}_X(\overline{HO}\cap\overline{S}\oplus\mathbb{F}\cdot h)\oplus
\mathrm{span}_{\mathbb{F}}\{(\mathrm{ad}_X(\partial_i))^{p^{k_i}}| i\in\mathbf{Y_0}, \quad 1\leq k_i< t_i\},$$
where  $\mathrm{ad}h$ is the degree derivation of $X$. Moreover,
$$\mathrm{Der}(\frak{g})\cong \mathrm{Der}(\frak{g}^{(1)})\cong \mathrm{Der}(\frak{g}^{(2)}).$$
\end{theorem}
\begin{proof}
Consider  $\varphi\in \mathrm{Der}_t(X)$
where $t\geq 0$. By virtue of \cite[Proposition 2.4]{wz} there exists an element $E\in \mathrm{Nor}_W(X)\cap W_t$
such that
$$(\varphi-\mathrm{ad}E)|_X=0.$$
Then the first part of the assertion holds from  Proposition \ref{t3.3} and Remark \ref{r3}.

Define
\begin{eqnarray*}
\varphi_1: \mathrm{Der}(\frak{g})&\longrightarrow &\mathrm{Der}(\frak{g}^{(1)})\\
\quad\quad\quad \phi &\longmapsto & \phi |_{\frak{g}^{(1)}},
\end{eqnarray*}
for $\phi\in\mathrm{Der}(\frak{g})$.
Obviously, $\varphi_1$ is a monomorphism.
Suppose  $\phi|_{\frak{g}^{(1)}}=0$ for any $\phi\in \mathrm{Der}(\frak{g}).$
We obtain
$$[\phi(A), (\frak{g}^{(1)})_{-1}]\subset \phi[A, \frak{g}^{(1)}]=0$$
for any $A \in \frak{g}$.
Then $\phi(A)\in (\frak{g}^{(1)})_{-1}$.
Note that $[\phi(A), (\frak{g}^{(1)})_{0}]\subset \phi[A, \frak{g}^{(1)}]=0.$
Hence we have $\phi(A)=0$
and $\mathrm{Der}(\frak{g})\cong \mathrm{Der}(\frak{g}^{(1)})$.
Similarly, $\mathrm{Der}(\frak{g}^{(1)})\cong \mathrm{Der}(\frak{g}^{(2)}).$
\end{proof}

\section{Outer derivations}
Let $X=\frak{g}$, $\frak{g}^{(1)}$, or $\frak{g}^{(2)}$. We denote
 the outer derivation algebras $\mathrm{Der_{out}}(X):=\mathrm{Der}(X)/\mathrm{ad}(X),$
which will be determined in this section. Recall $\mathrm{ad}h$ is the $\mathbb{Z}$-degree derivation of $X$,
 where $h=\Sigma_{i\in \mathbf{Y_0}}x_i\partial_i$. For future reference, we state the following results.
\begin{lemma}\label{l4.1}
The following statements hold in $\mathrm{Der}\mathcal{O}(m,m;\underline{t})$:

\item[$\mathrm{(1)}$] $[\partial^{p^{k_i}}_i, \partial^{p^{k_j}}_j]=0, \quad i,j \in \mathbf{Y_0}, \quad 1 \leq k_i<t_i, 1 \leq k_j<t_j$;

\item[$\mathrm{(2)}$] $[h, \partial_i^{p^{k_i}}]=0, \quad i\in \mathbf{Y_0}, \quad 1\leq k_i<t_i$;

\item[$\mathrm{(3)}$] $[\partial^{p^{k_i}}_i, \overline{HO}]\subset \overline{HO}, \quad i\in \mathbf{Y_0}, \quad 1\leq k_i<t_i$;

\item[$\mathrm{(4)}$] $[h, \overline{HO}]\subset \overline{HO};$

\item[$\mathrm{(5)}$] $[\partial^{p^{k_i}}_i, \mathrm{T_H}(a)]=\mathrm{T_H}(\partial^{p^{k_i}}_i(a)),
 \quad i\in \mathbf{Y_0}, \quad 1\leq k_i<t_i, \quad a \in \mathcal{O}(m,m;\underline{t});$

\item[$\mathrm{(6)}$] $[\partial^{p^{k_i}}_i, \overline{S}]\subset S', \quad i\in \mathbf{Y_0}, \quad 1\leq k_i<t_i$;

 \item[$\mathrm{(7)}$] $[h, \overline{S}]\subset S';$

\item[$\mathrm{(8)}$] $[h, \mathrm{T_H}(x_jx_{j'})]=0, \quad j\in \mathbf{Y_0}$.

\end{lemma}
\begin{proof}
(1)--(5) are the direct consequences  of \cite[Lemma 16]{lzw}, (6)--(8) are obvious.
\end{proof}

\begin{lemma}\label{l4.2}
The centralizers of $X$  in $W$ are zero.
%$E\subset W(m,m;\underline{t})$ and $\mathrm{ad}E(X)=0,$ then $E=0$.
where $X=\frak{g}$, $\frak{g}^{(1)}$, or $\frak{g}^{(2)}.$
\end{lemma}
\begin{proof}
Suppose  $E=\sum_{i=1}^{2m}a_i\partial_i$ is a centralizer of $X$ in $W$,  where $a_i \in \mathcal{O}(m,m;\underline{t})$.
Since
$$[E, \partial_j]=0, \, \mbox{for  all}\, \, j\in \mathbf{Y},$$
we have $a_i\in \mathbb{F}$.
Since
$$[E,\mathrm{T_H}(x_lx_{l'}-x_kx_{k'})]=0, \, \mbox{for  all}\, \,l, k\in \mathbf{Y_0}, l\neq k,$$
we have $a_l=a_{l'}=a_k=a_{k'}=0$. Hence $E=0$.
\end{proof}
Now we establish the relationship between $\mathrm{Der_{out}}(X)$
 and the quotient algebra $(\overline{HO}\cap \overline{S})/X.$

\begin{theorem}\label{t4.3}
The outer derivation algebras
$$\mathrm{Der_{out}}(X)\cong L\oplus (\overline{HO}\cap \overline{S})/X,$$
where
\begin{eqnarray*}
L&:=&\big(\mathbb{F}\cdot h\oplus
\mathrm{span}_{\mathbb{F}}\{(\mathrm{ad}_X(\partial_i))^{p^{k_i}}\mid i\in\mathbf{Y_0},
\quad 1\leq k_i< t_i\}\oplus X\big)\big/X\\
&\cong& \mathbb{F}\cdot h\oplus
\mathrm{span}_{\mathbb{F}}\{(\mathrm{ad}_X(\partial_i))^{p^{k_i}}\mid i\in\mathbf{Y_0},
\quad 1\leq k_i< t_i\}
\end{eqnarray*}
  is an Abelian Lie superalgebra. Moreover $(\overline{HO}\cap \overline{S})/X$
is an ideal of $\mathrm{Der_{out}}(X)$, where $X=\frak{g},$ $\frak{g}^{(1)}$, or $\frak{g}^{(2)}.$
\end{theorem}
\begin{proof}
It is similar to  \cite[Theorem 18]{lzw}.
\end{proof}

\begin{remark}\label{r4}
Suppose  $M$, $N_1$, $N_2$ are  subspaces of a vector space $V$. If $N=N_1\oplus N_2$
and $N_1\subset M$, then $M\cap N=N_1\oplus M\cap N_2 $.
\end{remark}

\begin{proposition}\label{p4.4}
$\overline{HO}\cap \overline{S}=\frak{g}\oplus \mathbb{F}\mathrm{T_H}(x_jx_{j'})\oplus
\mathrm{span}_{\mathbb{F}}\big\{x^{(\pi_i\varepsilon_i)}\partial_{i'}\big|i\in \mathbf{Y_0}\big\},$
for any  $j\in\mathbf{Y_0}$.
\end{proposition}
\begin{proof}
At first, we assert that for any $j\in\mathbf{Y_0}$,
$$\overline{S}=S'\oplus\mathbb{F}\mathrm{T_H}(x_jx_{j'}).$$
Note that
$$\mathrm{div}\big(\mathrm{T_H}(x_jx_{j'})\big)=-2, \, j\in\mathbf{Y_0}.$$
It is sufficient to show that $\overline{S}\subset S'\oplus\mathbb{F}\mathrm{T_H}(x_jx_{j'})$.
For any $A\in \overline{S}$, there exists $a\in \mathbb{F}$ such that $\mathrm{div}(A)=a$.
Hence  $\mathrm{div}(A+\frac a2\mathrm{T_H}(x_jx_{j'}))=0$ and $A\in S'\oplus\mathbb{F}\mathrm{T_H}(x_jx_{j'})$.

From \cite[Proposition 20]{lzw}, we know that
$$\overline{HO}=HO\oplus\mathrm{span}_{\mathbb{F}}
\big\{x^{(\pi_i\varepsilon_i)}\partial_{i'}\mid i\in \mathbf{Y_0}\big\},$$
hence
$$\overline{HO}\cap\overline{S}=HO\cap\overline{S}\oplus\mathrm{span}_{\mathbb{F}}
\big\{x^{(\pi_i\varepsilon_i)}\partial_{i'}\mid i\in \mathbf{Y_0}\big\}.$$
Moreover,
$$\overline{HO}\cap \overline{S}=\frak{g}\oplus \mathbb{F}\mathrm{T_H}(x_jx_{j'})\oplus
\mathrm{span}_{\mathbb{F}}\big\{x^{(\pi_i\varepsilon_i)}\partial_{i'}\mid i\in \mathbf{Y_0}\big\}$$
for any $j\in\mathbf{Y_0}$.
\end{proof}

\begin{corollary}\label{c4.5}
$\mathrm{Der_{out}}(\frak{g})\cong L\oplus\big(\frak{g}\oplus \mathbb{F}\mathrm{T_H}(x_jx_{j'})\oplus
\mathrm{span}_{\mathbb{F}}\big\{x^{(\pi_i\varepsilon_i)}\partial_{i'}
\mid i\in \mathbf{Y_0}\big\}\big)\big/\frak{g},$
for any $j\in\mathbf{Y_0},$
where $$\big(\frak{g}\oplus
\mathrm{span}_{\mathbb{F}}\big\{x^{(\pi_i\varepsilon_i)}\partial_{i'}
\mid i\in \mathbf{Y_0}\big\}\big)\big/\frak{g}$$
is just the odd part of $\mathrm{Der_{out}}(\frak{g})$ and
$$L\oplus\big(\mathbb{F}\mathrm{T_H}(x_jx_{j'})\oplus\frak{g}\big)\big/\frak{g}$$
is the even part.
\end{corollary}
\begin{proof}
This is a direct consequence of Theorem \ref{t4.3} and Proposition \ref{p4.4}.
\end{proof}
Let $A$, $B$ be Lie superalgebras. Recall that $A\ltimes_\varphi B$ is the
semidirect product of $A$ and $B$ with a homomorphism $\varphi: A\rightarrow \mathrm{Der}_{\mathbb{F}}B$.
Let $\mathrm{id}$ denote $a\mapsto a$, $a\in A$ when $A$ is the subalgebra of $\mathrm{Der}_{\mathbb{F}}B$.

Put $\iota:=\Sigma_{i=1}^{m}t_i-m$. Let $G:=G_{\bar{0}}\oplus G_{\bar{1}}$
be a $\mathbb{Z}_2$-graded vector space over $\mathbb{F}$, and $\{h_{-1}, h_0, h_1,\ldots, h_{\iota}\}$
 be an $\mathbb{F}$-basis of $G_{\bar{0}}$, $\{g_1,\ldots, g_m\}$ be an $\mathbb{F}$-basis of $G_{\bar{1}}$.
 Then $G$ is a $(\Sigma_{i=1}^{m}t_i+2)$-dimensional Lie superalgebra by means of
 \begin{eqnarray*}
 &&[h_{-1}, g_j]=0, \, j=1,\ldots, m;\\
 &&[h_{0}, g_j]=-2g_j, \, j=1,\ldots, m;\\
 &&[h_{i}, g_j]=0, \, i=1,\ldots, \iota, \, j=1,\ldots, m;\\
 &&[h_i, h_k]=[g_j, g_l]=0, \, i, k=-1,\ldots, \iota; \, j, l=1,\ldots, m.
 \end{eqnarray*}
Hence we have
$$G=\mathcal{C}(G)\oplus\mathbb{F}h_0\oplus G_{\bar{1}},$$
where $\mathcal{C}(G)$ is the center of $G$ and\\

(1) $\mathcal{C}(G)\oplus\mathbb{F}h_0$ is the even part of $G$;\\

(2) $G_{\bar{1}}$ is an Abelain subalgebra of $G$;\\

 Obviously, $\mathrm{ad}(\mathcal{C}(G)\oplus\mathbb{F}h_0)$ is
 a subalgebra of $\mathrm{Der}(G_{\bar{1}}).$ We can obtain
 $$G\cong \mathrm{ad}(\mathcal{C}(G)\oplus\mathbb{F}h_0)\ltimes_{\mathrm{id}} G_{\bar{1}}.$$
\begin{theorem}\label{t4.6}
The outer derivation algebra
$\mathrm{Der_{out}}(\frak{g})$ is isomorphic to
the Lie superalgebra $G$.
\end{theorem}
\begin{proof}
By Corollary \ref{c4.5} we have
$$\mathrm{Der_{out}}(\frak{g})\cong L\oplus\big(\frak{g}\oplus \mathbb{F}\mathrm{T_H}(x_jx_{j'})\oplus
\mathrm{span}_{\mathbb{F}}\big\{x^{(\pi_i\varepsilon_i)}\partial_{i'}
\mid i\in \mathbf{Y_0}\big\}\big)\big/\frak{g},$$
for any $j\in\mathbf{Y_0}.$ From Lemma \ref{l4.1}
we know that both $\big(\frak{g}\oplus
\mathrm{span}_{\mathbb{F}}\big\{x^{(\pi_i\varepsilon_i)}\partial_{i'}
\mid i\in \mathbf{Y_0}\big\}\big)\big/\frak{g}$ and
$L\oplus\big(\mathbb{F}\mathrm{T_H}(x_jx_{j'})\oplus\frak{g}\big)\big/\frak{g}$
are  Abelian.
A direct computation  shows that
\begin{eqnarray*}
 &&[\mathrm{T_H}(x_ix_{i'}), x^{(\pi_j\varepsilon_j)}\partial_{j'}]=0\quad i, j\in \mathbf{Y_0};\\
 &&[h_, x^{(\pi_j\varepsilon_j)}\partial_{j'}]=-2x^{(\pi_j\varepsilon_j)}\partial_{j'}\quad j\in \mathbf{Y_0};\\
 &&[\partial_i^{p^{k_i}}, x^{(\pi_j\varepsilon_j)}\partial_{j'}]=
 \delta_{i=j}x^{((\pi_j-p^{k_j})\varepsilon_j)}\partial_{j'}\\
 &&\quad \quad\quad\quad\quad\quad\quad\quad\in HO\cap S'=\frak{g}\quad i, j\in \mathbf{Y_0} \; 1\leq k_i< t_i.\\
 \end{eqnarray*}
 Now one can easily establish an isomorphism  from  $\mathrm{Der_{out}}(\frak{g})$
 to $G$.
\end{proof}
Now we consider the relationship among $\mathrm{Der_{out}}(\frak{g})$,
$\mathrm{Der_{out}}(\frak{g}^{(1)})$ and $\mathrm{Der_{out}}(\frak{g}^{(2)})$.
\begin{proposition}\label{t4.7} The following statements hold:
\begin{eqnarray*}
&&\mathrm{Der_{out}}(\frak{g}^{(1)})\cong
\mathrm{Der_{out}}(\frak{g})\oplus\frak{g}/\frak{g}^{(1)},\\
&&\mathrm{Der_{out}}(\frak{g}^{(2)})\cong
\mathrm{Der_{out}}(\frak{g}^{(1)})\oplus\frak{g}^{(1)}/\frak{g}^{(2)}.
\end{eqnarray*}
\end{proposition}
\begin{proof}
Define
\begin{eqnarray*}
\phi: L\oplus (\overline{HO}\cap \overline{S})/\frak{g}^{(1)}&\longrightarrow &
L\oplus (\overline{HO}\cap \overline{S})/\frak{g}\\
A+\frak{g}^{(1)} &\longrightarrow & A+\frak{g}.
\end{eqnarray*}
Note that $\phi$ is a monomorphism, and
$\ker(\phi)=\frak{g}/\frak{g}^{(1)}$.
By Theorem \ref{t4.3} we obtain
$$\mathrm{Der_{out}}(\frak{g}^{(1)})\big/\ker(\phi)\cong\mathrm{Der_{out}}(\frak{g}).$$
Applying  Corollary  \ref{c4.5} we can obtain, for any $j\in\mathbf{Y_0}$,
\begin{eqnarray*}
\mathrm{Der_{out}}(\frak{g})&\cong&L\oplus\big(\frak{g}\oplus \mathbb{F}\mathrm{T_H}(x_jx_{j'})\oplus
\mathrm{span}_{\mathbb{F}}\big\{x^{(\pi_i\varepsilon_i)}\partial_{i'}
\big|i\in \mathbf{Y_0}\big\}\big)\big/\frak{g}\\
&\cong&L\oplus\big(\frak{g}^{(1)}\oplus \mathbb{F}\mathrm{T_H}(x_jx_{j'})\oplus
\mathrm{span}_{\mathbb{F}}\big\{x^{(\pi_i\varepsilon_i)}\partial_{i'}
\big|i\in \mathbf{Y_0}\big\}\big)\big/\frak{g}^{(1)},
\end{eqnarray*}
which is a subalgebra of $\mathrm{Der_{out}}(\frak{g}^{(1)})$ by means of Lemma \ref{l4.1}.
Hence we obtain
$$\mathrm{Der_{out}}(\frak{g}^{(1)})\cong
\mathrm{Der_{out}}(\frak{g})\oplus\frak{g}/\frak{g}^{(1)}.$$
Similarly
$$\mathrm{Der_{out}}(\frak{g}^{(2)})\cong
\mathrm{Der_{out}}(\frak{g}^{(1)})\oplus\frak{g}^{(1)}/\frak{g}^{(2)}.$$
\end{proof}

Put $||u||=|\pi|-(\Sigma_{i'\in u}\pi_i)+|u|-2$, $u\in\mathbb{B}(m)$, $i\in \mathbf{Y_0}.$
Now we can define $G^1:=G\oplus \Lambda(m)$,
where  $\Lambda(m)$ is the  exterior superalgebra. $G'$ has a $\mathbb{Z}_2$-grading structure induced
 by the $\mathbb{Z}_2$-grading structures of $G$ and $\Lambda(m)$, then $G^1$ is a  $(\Sigma_{i=1}^{m}t_i+2+2^m)$-dimensional Lie superalgebra by means of
 \begin{eqnarray*}
 &&[h_{-1}, x^u]=x^u;\\
 &&[h_{0}, x^u]=||u||x^u;\\
 &&[h_{i}, x^u]=0, \, \,i=1,\ldots, \iota;\\
 &&[g_i, x^u]=(-1)^{(i, u)}\delta_{i'\in u}x^{u-\langle i'\rangle}, \,
 i=1,\ldots, m;\\
 &&[x^u, x^v]=0,
 \end{eqnarray*}
for all $x^u, x^v\in \Lambda(m),$
and $(-1)^{(i, u)}$ is determined by the equation
$\partial_{i'}(x^u)=(-1)^{(i, u)}x^{u-\langle i'\rangle}.$
Obviously, $\mathrm{ad}(G)$
is a subalgebra of $\mathrm{Der}(\Lambda(m))$ and
$$G^1\cong\mathrm{ad}(G)\ltimes_{\mathrm{id}} \Lambda(m).$$

Recall $\mathfrak{A}_{1}=\{\mathrm{T_H}(x^{(\alpha)}x^{u})\mid
\mathbf{I}(\alpha,u)=\widetilde{\mathbf{I}}(\alpha,u)=\emptyset\}$ is a $\mathbb{Z}_2$-graded subspace of $\frak{g}$ with
$\mathfrak{A}_{1}=\mathfrak{A}_{1\bar{0}}\oplus\mathfrak{A}_{1\bar{1}},$
where
$\mathfrak{A}_{1\bar{\theta}}=\mathfrak{A}_{1}\cap\frak{g}_{\bar{\theta}}, \, \bar{\theta}=\bar{0}, \bar{1}.$
Notice that
$\frak{g}=\frak{g}^{(1)}\oplus\mathfrak{A}_{1}$.
\begin{theorem}\label{t4.8}
The outer derivation algebra
$\mathrm{Der_{out}}(\frak{g}^{(1)})$
is isomorphic to the Lie superalgebra $G^1$.
\end{theorem}
\begin{proof}
From Proposition  \ref{t4.7} we know that $\mathrm{Der_{out}}(\frak{g}^{(1)})\cong
\mathrm{Der_{out}}(\frak{g})\oplus\frak{g}/\frak{g}^{(1)}$.
Notice that $\frak{g}^{(1)}$ is an ideal of $\frak{g}$. Applying Lemma \ref{l4.1} and Theorem \ref{t4.6},
it is sufficient to consider the operation between  $\mathrm{Der_{out}}(\frak{g})$ and $\mathfrak{A}_{1}$.
By the definition of $\mathfrak{A}_{1}$ we know $\mathfrak{A}_{1}$ is spanned by the elements with the form
 $\mathrm{T_H}(x^{(\pi-(\Sigma_{r'\in u}\pi_r\varepsilon_r))}x^u)$,
a direct computation  shows that
\begin{eqnarray*}
 &&\big[\mathrm{T_H}(x_ix_{i'}), \mathrm{T_H}\big(x^{(\pi-(\Sigma_{r'\in u}\pi_r\varepsilon_r))}x^u\big)\big]
=\mathrm{T_H}\big(x^{(\pi-(\Sigma_{r'\in u}\pi_r\varepsilon_r))}x^u\big), \,i\in \mathbf{Y_0};\\
 &&\big[h_, \mathrm{T_H}\big(x^{(\pi-(\Sigma_{r'\in u}\pi_r\varepsilon_r))}x^u\big)\big]=
 ||u||\mathrm{T_H}\big(x^{(\pi-(\Sigma_{r'\in u}\pi_r\varepsilon_r))}x^u\big);\\
 &&\big[\partial_i^{p^{k_i}}, \mathrm{T_H}\big(x^{(\pi-(\Sigma_{r'\in u}\pi_r\varepsilon_r))}x^u\big)\big]
 \subset\frak{g}^{(1)},\, i\in \mathbf{Y_0}, \; 1\leq k_i< t_i;\\
 &&\big[x^{(\pi_i\varepsilon_i)}\partial_{i'}, \mathrm{T_H}\big(x^{(\pi-(\Sigma_{r'\in u}\pi_r\varepsilon_r))}x^u\big)\big]=
 \delta_{i'\in u}\mathrm{T_H}\big(x^{(\pi-(\Sigma_{r'\in u}\pi_r\varepsilon_r)+\pi_i\varepsilon_i)}\partial_{i'}x^u\big), \, i\in \mathbf{Y_0}.
 \end{eqnarray*}
 Note that $[\mathfrak{A}_{1}, \mathfrak{A}_{1}]\subset\frak{g}^{(1)}.$
 Hence we can easily establish an isomorphism  from  $\mathrm{Der_{out}}(\frak{g}^{(1)})$
 to $G^1$.
\end{proof}

Put $G^2:=G^1\oplus\mathbb{F}f$. Define the even part of $G^2$ as follows:
\begin{eqnarray*}
&&G^2_{\bar{0}}=G^1_{\bar{0}}\oplus\mathbb{F}f\, \, \, \mbox{if $m$ is even},\\
&&G^2_{\bar{0}}=G^1_{\bar{0}}\, \, \, \quad\quad\,\,\mbox{if $ m$ is odd}.
\end{eqnarray*}
Then $G^2$ is a $(\Sigma_{i=1}^{m}t_i+3+2^m)$-dimensional Lie superalgebra by means of
 \begin{eqnarray*}
 &&[h_{-1}, f]=2f;\\
 &&[h_{0}, f]=-4f;\\
 &&[h_{i}, f]=0, \, i=1,\ldots, \iota;\\
 &&[g_i, f]=0, \, i=1,\ldots, m; \\
 &&[x^u, f]=0, \, x^u\in \Lambda(m).
 \end{eqnarray*}
Moreover, $G^2\cong\mathrm{ad}(G^1)\ltimes_{\mathrm{id}} \mathbb{F}f.$

Recall $\frak{g}^{(1)}=\frak{g}^{(2)}\oplus (\frak{g}^{(1)})_{\xi-4}$ and
$$(\frak{g}^{(1)})_{\xi-4}=\mathrm{span}_{\mathbb{F}}\bigg\{\mathrm{T_H}\bigg(x^{(\pi-\varepsilon_i)}x^{\omega-\langle i' \rangle}-
\sum_{r\in\mathbf{Y_0}\backslash \{i\}}\Gamma^i_r \big(x^{(\pi-\varepsilon_i)}x^{\omega-\langle i' \rangle}\big)\bigg)\, \bigg|\, i\in \mathbf{Y_0}\bigg\}.$$
Notice that $\mathrm{dim}(\frak{g}^{(1)})_{\xi-4}=1.$
\begin{theorem}\label{t4.9}
The outer derivation algebra $\mathrm{Der_{out}}(\frak{g}^{(2)})$
is isomorphic to the Lie superalgebra $G^2$.
\end{theorem}
\begin{proof}
Similar to Theorem \ref{t4.8}, it is sufficient to consider the operation between $\mathrm{Der_{out}}(\frak{g}^{(1)})$ and $(\frak{g}^{(1)})_{\xi-4}$.
For any $$M=a\mathrm{T_H}\bigg(x^{(\pi-\varepsilon_i)}x^{\omega-\langle i' \rangle}-
\sum_{r\in\mathbf{Y_0}\backslash \{i\}}\Gamma^i_r \big(x^{(\pi-\varepsilon_i)}x^{\omega-\langle i' \rangle}\big)\bigg)
\in (\frak{g}^{(1)})_{\xi-4},$$
where $a\in \mathbb{F}$,
we can obtain:
\begin{eqnarray*}
 &&[\mathrm{T_H}(x_ix_{i'}), M]=2M, \,i\in \mathbf{Y_0};\\
 &&[h_, M]=-4M;\\
 &&[\partial_j^{p^{k_j}}, M]=\mathrm{T_H}\bigg(x^{(\pi-p^{k_j}\varepsilon_j-\varepsilon_i)}x^{\omega-\langle i' \rangle}-
\sum_{r\in\mathbf{Y_0}\backslash \{i\}}\Gamma^i_r \big(x^{(\pi-p^{k_j}\varepsilon_j-\varepsilon_i)}x^{\omega-\langle i' \rangle}\big)\bigg)\\
&&\in\frak{g}^{(2)},\, j\in \mathbf{Y_0};\\
&&[x^{(\pi_i\varepsilon_i)}\partial_{i'}, M]=0, \, i\in \mathbf{Y_0}.
\end{eqnarray*}
 Note that
$[\mathfrak{A}_{1}, (\frak{g}^{(1)})_{\xi-4}]=0.$
Hence we can easily establish an isomorphism  from  $\mathrm{Der_{out}}(\frak{g}^{(2)})$
 to $G^2$.
\end{proof}

\end{document}